\documentclass[aap,noinfoline]{imsart}
\pdfoutput=1
\usepackage[top=1.2in,bottom=1.2in,left=1.25in,right=1.25in]{geometry}
\setattribute{author}{prefix}{\empty}
\setattribute{copyright}{text}{\empty}
\usepackage[usenames,dvipsnames]{color}
\definecolor{RefColor}{rgb}{0,0,.65}
\RequirePackage[colorlinks,linkcolor=RefColor,citecolor=RefColor,urlcolor=RefColor]{hyperref}
\hypersetup{pdfinfo={Subject={ }}}
\RequirePackage[OT1]{fontenc}
\usepackage{amsmath,amssymb,amscd,amsfonts,amsthm,mathtools}
\usepackage[numbers]{natbib}

\RequirePackage[OT1]{fontenc}
\RequirePackage{amsthm,amsmath}
\usepackage[usenames,dvipsnames]{color}
\RequirePackage[colorlinks,linkcolor=blue,citecolor=blue]{hyperref}

\usepackage{thmtools}
\usepackage[capitalize]{cleveref}
\usepackage{booktabs}
\usepackage{tikz}
\usetikzlibrary{shapes.geometric,calc,arrows,chains,matrix,positioning,scopes,decorations.markings,decorations.pathreplacing,intersections,backgrounds,fit,scopes}
\usepackage{dsfont}

\startlocaldefs
\theoremstyle{plain}

\newtheorem{theorem}{Theorem}
\newtheorem{lemma}[theorem]{Lemma}
\newtheorem{proposition}[theorem]{Proposition}
\newtheorem{corollary}[theorem]{Corollary}

\newtheorem*{fact*}{Fact}
\theoremstyle{definition}

\newtheorem{remark}[theorem]{Remark}

\declaretheoremstyle[postheadspace=\newline,
  mdframed={backgroundcolor=gray!10!white, 
    hidealllines=true, 
    innertopmargin=2pt, 
    innerbottommargin=4pt, 
    skipabove=8pt,
    skipbelow=10pt,
  nobreak=true}
]{grayboxed} 
\declaretheorem[style=plain,qed=$\triangleleft$]{auxtheorem}
\declaretheorem[style=grayboxed,sibling=auxtheorem]{algorithm}
\declaretheorem[style=grayboxed,name=Algorithm,sibling=auxtheorem]{algorithmdash}

\declaretheoremstyle[mdframed={backgroundcolor=gray!10!white, 
    hidealllines=true, 
    innertopmargin=2pt, 
    innerbottommargin=1pt, 
    skipabove=3pt,
    skipbelow=10pt,
  nobreak=true}
]{modgrayboxed} 

\tikzstyle{mybraces}=[mirrorbrace/.style={
          decoration={brace, mirror},
          decorate},brace/.style={
          decoration={brace},
          decorate}]
\endlocaldefs

\makeatletter
\newcommand{\eqnum}{\leavevmode\hfill\refstepcounter{equation}\textup{\tagform@{\theequation}}}
\makeatother

\begin{document}
\begin{frontmatter}
  \title{Subsampling large graphs and invariance in networks}
  \author{\fnms{Peter} \snm{Orbanz}\corref{}\ead[label=e2]{porbanz@stat.columbia.edu}}
  \affiliation{Columbia University}
    \address{Department of Statistics\\
      1255 Amsterdam Avenue\\
      New York, NY 10027, USA\\
      \printead{e2}
    }    
  \begin{abstract}
    Specify a randomized algorithm that, given a very large graph or network, extracts a random subgraph. What can we learn about the input graph from a single subsample? We derive laws of large numbers for the sampler output, by relating randomized subsampling to distributional invariance: Assuming an invariance holds is tantamount to assuming the sample has been generated by a specific algorithm. That in turn yields a notion of ergodicity. Sampling algorithms induce model classes---graphon models, sparse generalizations of exchangeable graphs, and random multigraphs with exchangeable edges can all be obtained in this manner, and we specialize our results to a number of examples. One class of sampling algorithms emerges as special: Roughly speaking, those defined as limits of random transformations drawn uniformly from certain sequences of groups. Some known pathologies of network models based on graphons are explained as a form of selection bias.
  \end{abstract}
  \maketitle
\end{frontmatter}

\newcommand{\Alg}[2]{S_{#1\rightarrow #2}}
\newcommand{\AlgInf}[1]{S_{#1}}
\def\PAlg{\mathbb{P}_{\!\ind{alg}}}
\def\U{\mathbb{U}}

\def\condind{{\perp\!\!\!\perp}}
\def\ie{i.e.\ }
\def\eg{e.g.\ }
\def\kword#1{\textbf{#1}}

\def\pull#1{{#1}^{\text{\fontsize{.4em}{.5em}\selectfont\rm \#}}}
\def\push#1{{#1}_{\!\text{\fontsize{.4em}{.5em}\selectfont\rm \#}}}

\def\abstspace{\Omega}
\def\abstmeasure{\mathbb{P}}
\def\xspace{\mathbf{X}}
\def\tspace{\mathbf{T}}
\def\sspace{\mathbf{S}}
\def\gspace{\mathbf{G}}
\def\gset{\mathcal{G}}
\def\gsetall{\gset_{\ast}}

\def\Law{\mathcal{L}}
\def\mean{\mathbb{E}}
\def\SG{\textsc{Sym}}
\def\SGinf{\mathbb{S}_{\infty}}

\def\INV{\textsc{Inv}}

\def\ex{\text{\rm ex}}
\def\muP{\mu_{\ind{P}}}
\def\nuP{\nu_{\ind{P}}}
\def\pr{\textrm{pr}}
\def\Pr{\textrm{Pr}}
\def\braces#1{{\lbrace #1 \rbrace}}
\def\bigbraces#1{{\big\lbrace #1 \big\rbrace}}
\def\Bigbraces#1{{\Big\lbrace #1 \Big\rbrace}}
\def\indfunction{\mathbb{I}}
\def\ddisc{d_{\mbox{\tiny 0-1}}}
\def\gsetN{\gset_{\mathbb{N}}}
\def\balls{\mathbf{B}}
\def\ball{B}

\def\tkernel{\mathbf{t}}
\def\pMeas{\mathbf{PM}}
\def\borel{\mathcal{B}}
\newcommand{\argdot}{{\,\vcenter{\hbox{\tiny$\bullet$}}\,}}
\def\ind#1{\text{\rm\tiny #1}}
\def\equivS{\equiv_{\ind{S}}}
\def\classS#1{[#1]_{\ind{S}}}
\def\Aut{\textrm{Aut}}
\def\Aut{\text{\sc Aut}}
\def\sigmainv{\sigma_{\ind{S}}}
\def\bsigmainv{\overline{\sigma}_{\ind{IN\!V}}}
\def\suffkernel{\mathbf{p}}

\def\equdist{\stackrel{\text{\rm\tiny d}}{=}}
\def\equas{=_{\text{\rm\tiny a.s.}}}
\def\weakly{\stackrel{\text{\rm\tiny w}}{\longrightarrow}}

\def\empavg{\mathbb{F}}
\def\Haar{\lambda}
\def\Folner{F{\o}lner }

\def\rest[#1]#2{#2\big\vert_{#1}}
\def\restk#1{\rest[k]{#1}}
\def\restn#1{\rest[n]{#1}}
\def\textrest[#1]#2{#2\vert_{_{#1}}}
\def\textrestn#1{\textrest[n]{#1}}
\def\textrestk#1{\textrest[k]{#1}}
\def\bitstring{\gamma}
\def\Bitstring{\Gamma}
\def\bitlaw{\mathbb{Q}}
\def\ERG{\mathcal{E}}
\def\ERG{{\text{\sc Erg}}}

\def\abstfield{\mathcal{A}}
\def\Uniform{\text{\rm Uniform}}
\def\Bernoulli{\text{\rm Bernoulli}}
\def\simiid{\sim_{\mbox{\tiny iid}}}
\def\W{\mathbf{W}}
\def\hd{\tau}

\def\iid{i.i.d.\ }

\def\sbgraphon{v}
\def\Sbgraphon{V}
\def\sbgraphonspace{\mathbf{V}}
\def\sbquotientspace{\hat{\mathbf{V}}}
\def\graphon{w}
\def\Graphon{W}
\def\mpartition{\mathbf{s}}
\def\Mpartition{\mathbf{S}}
\def\mpartitionspace{\mathcal{S}}
\def\indfunction{\mathbb{I}}

\def\graphonspace{\mathbf{W}}
\def\kernel{\mathbf{p}}

\def\Erdos{Erd\H{o}s}
\def\Renyi{R\'enyi}

\def\Aplus{\A^{\!\ind{+}}}
\def\A{\mathbf{A}}
\def\SGsub#1{\SG_{\ind{[}#1\ind{]}}}

\newcommand{\todo}[1]{\textcolor{red}{{\tt #1}}}
\renewcommand{\labelitemi}{{\raisebox{0.28ex}{\tiny$\bullet$}}}

\def\hom{p}
\def\P{\abstmeasure}
\def\mysetminus{\!\smallsetminus\!}
\def\IN{\mathbf{\mathbb{Y}}}
\def\OUT{\mathbf{\mathbb{X}}}
\def\subIN{\mathcal{Y}}
\def\gset{\mathcal{X}}
\def\sset{\mathcal{S}}
\def\sspace{\mathbf{S}}
\def\Lone{\mathbf{L}_1}
\def\equivIN{\equiv_{\ind{in}}}
\def\classIN#1{[#1]_{\ind{in}}}
\def\hset{\mathcal{H}}
\def\sp#1{\!\left< #1 \right>}
\def\model{\mathcal{P}}
\def\equivW{\stackrel{\textrm{\tiny w}}{\equiv}}
\def\equivS{\equiv_{\ind{out}}}
\def\equivS{\stackrel{\resizebox{!}{.24em}{\textsc{out}}}{\equiv}}
\def\equivIN{\equiv_{\ind{S}}}
\def\hom{t}
\def\REL{\textsc{Rel}}
\def\PRE{\textsc{Pre}}
\def\group{\mathbf{G}}
\def\emptyset{\varnothing}
\def\bbN{\mathbb{N}}

\def\inclusion{\iota}
\def\projA{\overline{\A}}

\def\Triag{\textsc{Triag}}
\def\Diag{\textsc{Diag}}
\def\deg{\text{\rm deg}}
\def\mult{\text{\rm mult}}
\def\rdeg{\overline{d}}
\def\rmult{\overline{m}}
\def\bbN{\mathbb{N}}
\def\edgeset{\textsc{Edges}}
\def\vertexset{\textsc{Vertices}}

\def\support{\textsc{Supp}}
\def\SigmaS{\Sigma}
\def\Null{\textsc{Null}}
\def\bSigmaS{\overline{\Sigma}}
\def\FSym{\mathbb{S}_{\mbox{\tiny F}}}
\def\Sym#1{\mathbb{S}_{#1}}
\def\equas{{\stackrel{\;\resizebox{!}{.24em}{\text{\rm a.s.}}\;}{=}}}
\def\fam{\mathcal{T}}
\def\bsigma{\overline{\sigma}}

\section{Introduction}

Consider a large graph or network, and invent a randomized algorithm that generates a subgraph. 
The algorithm can be understood as a model of an experimental design---a protocol used to collect data in a survey, 
or to sample data from a network---or as an actual program extracting data from a data base.
Use the algorithm to extract a sample graph, small relative to input size.
What information can be obtained from a single such sample? Certainly, that should depend on the algorithm. 
We approach the problem starting from a simple observation:
Fix a sequence $y_n$ with $n$ entries. Generate a random sequence
$X_k$ by sampling $k$ elements of $y_n$, uniformly and independently with replacement.
Then $X_k$ is exchangeable, and it remains so
under a suitable distributional limit ${n\rightarrow\infty}$ in input size.
Similarly, one can generate an exchangeable graph (a random graph whose law is invariant under 
permutations of the vertex set) by sampling vertices of
an input graph independently, and extracting the induced subgraph. 
The resulting class of random graphs is equivalent to graphon models 
\citep{Borgs:Chayes:Lovasz:Sos:Vesztergombi:2008,Borgs:Chayes:Lovasz:Sos:Vesztergombi:2012:1,Diaconis:Janson:2007}.
Exchangeability is an example of distributional symmetry, that is, invariance of a
distribution under a class of transformations \citep{Kallenberg:2005}. 
Thus, a randomized algorithm (independent selection of elements) induces a symmetry 
principle (exchangeability of elements), and when applied to graphs, it also induces a model class (graphon models).
The purpose of this work is to show how the interplay of these properties answers what can be learned from a 
single subgraph, both for the example above and for other algorithms.

\subsection{Overview}

Start with a large graph $y$. The graph may be unadorned, or a ``network'' in which
each edge and vertex is associated with some mark or observed value. We always assume that the ``initial
subgraph of size $n$'', denoted $\textrestn{y}$, is unambiguously defined. This may be the induced subgraph
on the first $n$ vertices (if vertices are enumerated), the subgraph incident to the first $n$ edges
(if edges are enumerated), the neighborhood of size $n$ around a fixed root, et cetera; details follow 
in \cref{sec:spaces}.
Now invent a randomized algorithm that generates a graph of size ${k<n}$
from the input ${\textrestn{y}}$, and denote this random graph $\Alg{n}{k}(y)$. For example:
\begin{algorithm}
  \label{alg:uvertex:simple}
  \begin{tabular}{rl}
    \text{{\small i.})} & \text{Select $k$ vertices of $\textrestn{y}$ 
      independently and uniformly without replacement.}\\
    \text{{\small ii.})} & \text{Extract the induced subgraph ${\Alg{n}{k}(\textrestn{y})}$ of 
      $\textrestn{y}$ on these vertices.}\\
    \text{{\small iii.})} & \text{Label the vertices of ${\Alg{n}{k}(\textrestn{y})}$ 
      by ${1,\ldots,k}$ in order of appearance.}\\
  \end{tabular}
\end{algorithm}
We assume $y$ is so large that it is modeled as infinite,
and hence ask for a limit in input size:
Is there a random variable ${\AlgInf{\infty}(y)}$ that can be regarded
as a sample of infinite size from an infinite input graph $y$? That is, a variable
such that, for each output
size $k$, the restriction ${\textrestk{\AlgInf{\infty}(y)}}$ is the distributional limit in input size $n$,
\begin{equation*}
  \Alg{n}{k}(y)\quad\xrightarrow{\;\text{\tiny d}\;}\quad\textrestk{\AlgInf{\infty}(y)}
  \qquad\qquad\text{ as }\quad n\rightarrow\infty\;.
\end{equation*}
Necessary and sufficient conditions are given in \cref{result:alg:limit}. 
By sampling and passing to the limit, some information about $y$ is typically lost---$y$ cannot
be reconstructed precisely from the sample output, which is of some importance to our purposes;
see \cref{remark:population}.

The main problem we consider is inference from a single realization:
By a \emph{model} on $\OUT$, we mean a set $\mathcal{P}$ of probability measures on $\OUT$.
Observed is a single graph $X_k$ of size $k$. We choose a sampling algorithm as a modeling assumption
on how the data was generated, and take its limit as above. For a given set $\subIN$ of input graphs,
the algorithm induces a model
\begin{equation*}
  \mathcal{P}:=\braces{P_y|y\in\subIN} \qquad\text{ where }\qquad P_y:=\Law(\AlgInf{\infty}(y))\;.
\end{equation*}
The observed graph $X_k$ is modeled as a sample ${X_k=\AlgInf{k}(y)}$
of size $n$ from an infinite input graph. Even as ${n\rightarrow\infty}$, this means that only
a single draw from the model distribution is available. When $k$ is finite, this is further constrained
to a partial realization. 

We have already observed the sampler output $\AlgInf{\infty}(y)$ is exchangeable for certain algorithms,
or more generally invariant under some family of transformations of $\OUT$.
A fundamental principle of ergodic theory is that invariance under a suitable family of transformations 
$\fam$ defines a set $\ERG(\fam)$ of distinguished probability measures, the \emph{$\fam$-ergodic} measures.
\cref{sec:symmetry} reviews the relevant concepts. The utility of invariance to our problem is that
\emph{ergodic measures of suitable transformation families are distinguishable
  by a single realization}:
If a model is chosen as a subset $\mathcal{P}$ of these $\fam$-ergodic measures, and a single
draw ${X\sim P}$ from some (unknown) element $P$ of $\mathcal{P}$ is observed, the distribution
$P$ is unambiguously determined by $X$. 
The following consequence is made precise in \cref{sec:sampling:symmetry}:
\begin{center}
\emph{If the limiting sampler can be shown to have a suitable invariance property, then
  inference from a single realization is possible in principle.}
\end{center}

We make the provision \emph{in principle} as the realization $\AlgInf{\infty}(y)$ only
determines $P_y$ abstractly---for an actual
algorithm, there is no obvious way to derive $P_y$ from an observed graph.
We hence ask under what conditions the expectations
\begin{equation*}
  \mean[f(\AlgInf{\infty}(y))]=P_y(f)\qquad\text{ for some }f\in\Lone(P_y)
\end{equation*}
can be computed or approximated given an observed realization. (Here and throughout,
we use the notation ${P(f)=\int fdP}$.)
\cref{lemma:main} provides the following answer:
If $\fam$ is specifically a group satisfying certain properties,
one can choose a specific sequence of finite subsets $\A_k$ of of $\fam$ and define
\begin{equation*}
  \empavg_k^x:=\frac{1}{|\A_k|}\sum_{t\in\A_k}\delta_{t(x)}
  \qquad\text{ hence }\qquad
  \empavg_k^x(f)=\frac{1}{|\A_k|}\sum_{t\in\A_k}f(t(x))\;.
\end{equation*}
The sequence $(\empavg_k^x)$ can be thought of as an empirical measure, and satisfies a law
of large numbers: If a sequence $(f_k)$ of functions converges almost
everywhere to a function $f$, then
\begin{equation}
  \label{intro:lln}
  \empavg_k^{\AlgInf{\infty}(y)}(f_k)\quad\xrightarrow{k\rightarrow\infty}\quad\mean[f(\AlgInf{\infty}(y))]
\end{equation}
under suitable conditions on the sampler.
\cref{lemma:main} formulates such convergence as a law of large numbers for symmetric random variables,
and subsumes several
known results on exchangeable structures, graphons, etc.

The graph $\AlgInf{\infty}(y)$ in \eqref{intro:lln} is typically infinite. To work with a 
finite sample $\AlgInf{k}(y)$, we formulate additional conditions on $\fam$ that let transformations
act on finite structures, which leads to a class of groups we call \emph{prefix actions}.
We then define sampling algorithms by randomly applying a transformation: Fix a prefix action $\fam$,
draw a random transformation $\Phi_n$ uniformly from
those elements of $\fam$ that affect only a subgraph of size $n$, and define
\begin{equation*}
  \Alg{n}{k}(y)=\textrestk{\Phi_n(\textrestn{y})}\;.
\end{equation*}
In words, randomly transform $\textrestn{y}$ using $\Phi_n$ and then truncate at output size ${k<n}$.
These algorithms turn out to be of particular interest: They induce various known
models---graphons, edge-exchangeable graphs, and others---and
generically satisfy $\fam$-invariance and other
non-trivial properties; see \cref{result:direct:union:sampler}. 
The law of large numbers strengthens to
\begin{equation}
  \label{intro:lln:finite}
  \empavg_k^{\AlgInf{k}(y)}(f_k)\quad\xrightarrow{k\rightarrow\infty}\quad\mean[f(y)]\;.
\end{equation}
In contrast to \eqref{intro:lln}, the approximation is now a function of a finite sample
of size $k$, and the right-hand side a functional of the input graph $y$, rather than of 
$\AlgInf{\infty}(y)$. See \cref{result:estimate:f:of:y}.

\begin{table}[t]
  \small
  \setlength{\abovecaptionskip}{-10pt}
  \setlength{\belowcaptionskip}{-5pt}
  \resizebox{\textwidth}{!}{
  \begin{tabular}{rlll}
    \toprule
    sec.
    & sampling scheme
    &
    symmetry
    &
    induced model class \\
    \midrule
    \ref{sec:graphons} & $k$ independent vertices
    &
    vertices exchangeable
    & 
    graphon models \citep{Borgs:Chayes:Lovasz:Sos:Vesztergombi:2008}
    \\[.3em]
    \ref{sec:KEGs} & select vertices by coin flips
    &
    underlying point process
    & 
    generalized graphons
    \\
    & + delete isolated vertices & exchangeable & 
    \citep{Borgs:Chayes:Cohn:Holden:2016:1,Veitch:Roy:2015:1}\\[.3em]
    \ref{sec:edge:exch:multigraphs} & $k$ independent edges 
    &
    edges exchangeable
    &
    ``edge-exchangeable graphs''\\
    & (input multigraph) & &     
    \citep{Crane:Dempsey:2016:1,Cai:Campbell:Broderick:2016:1}\\
    \ref{sec:benjamini:schramm}
    &
    neighborhood of random vertex
    &
    involution invariance \citep{Aldous:Lyons:2007:1}
    &
    certain local weak limits \citep{Benjamini:Schramm:2001:1}
    \\
    \bottomrule
    \label{tab:symmetry}
  \end{tabular}
  }
  \caption{}
  \label{tab:symmetries}
\end{table}

With the general results in place, we consider specific algorithms. 
In some cases, the algorithm induces a known class of random graphs as its
family ${\braces{P_y|y\in\subIN}}$ of possible output distributions; see
\cref{tab:symmetries}.
\cref{sec:sampling:vertices} concerns \cref{alg:uvertex:simple}, exchangeable graphs,
and graphon models. We consider modifications of \cref{alg:uvertex:simple}, and show how
known misspecification problems that arise when graphons are used to model network data can be
explained as a form of selection bias.
\cref{sec:partitions} relates well-known properties of exchangeable sequences and partitions to
algorithms sampling from fixed sequences and partitions.
That serves as preparation for \cref{sec:sampling:edges}, on algorithms that select a random sequence of edges, 
and report the subgraph incident to those edges.
If the input graph $y$ is simple, a property of $y$ can be estimated from the
sample output if and only if it is a function of the degree sequence of $y$.
If $y$ is a multigraph and the limiting relative multiplicities of its edges sum to 1, 
the algorithm generates edge-exchangeable graphs in the sense of 
\citep{Crane:Dempsey:2016:1,Cai:Campbell:Broderick:2016:1,Janson:2017:2}.
The two cases differ considerably: For simple input graphs, the sample output is completely determined
by vertex degrees, for multigraphs by edge multiplicities.
If a sampling algorithm explores a graph by following edges---as 
many actual experimental designs do, see e.g.\ \citep{Kolaczyk:2009:1}---the stochastic 
dependence between edges tends to
become more complicated, 
and understanding symmetries of such algorithms is much harder.
\cref{sec:neighborhoods} puts some previously known properties of methods algorithms 
that sample neighborhoods in the context of this work.

\subsection{Related work}
\hspace{-1em}\footnote{The ideas proposed here are used explicitly in forthcoming work of 
\citet{Veitch:Roy:2016:1} and \citet*{Borgs:Chayes:Cohn:Veitch:2017:1}, both already
available as preprints. Cf.\ \cref{sec:KEGs}.}\hspace{.5em}
Related previous work
largely falls into two categories: One concerns random graphs,
exchangeability, graph limits, and related topics. This work is mostly theoretical,
and intersects probability, combinatorics,
and mathematical statistics 
\citep{Borgs:Chayes:Lovasz:Sos:Vesztergombi:2008,Borgs:Chayes:Lovasz:Sos:Vesztergombi:2012:1,Diaconis:Janson:2007,
Kallenberg:2005,Ambroise:Matias:2012:1,Bickel:Chen:Levina:2011,Gao:Lu:Zhou:2015:1,Klopp:Tsybakov:Verzelen:2015:1,
Orbanz:Roy:2014}.
A question closely related to the problem considered here---what probabilistic symmetries aside from
exchangeability of vertices are applicable to networks analysis problems---was posed in \citep{Orbanz:Roy:2014}.
One possible solution, due to \citet{Caron:Fox:2014:1},
is to require exchangeability of an underlying point process. This idea can be used to 
generalize graph limits to sparse graphs
\citep{Veitch:Roy:2015:1,Borgs:Chayes:Cohn:Holden:2016:1}.
Another answer are random multigraphs whose edges, rather than vertices, are exchangeable
\citep{Crane:Dempsey:2016:1,Cai:Campbell:Broderick:2016:1,Janson:2017:2}. 
These are exchangeable partitions, in the sense of Kingman \citep{Kingman:1978:2},
of the upper triagonal of the set $\mathbb{N}^2$.
The second related category of work covers
experimental design in networks, and constitutes a substantial literature, see \citep{Kolaczyk:2009:1} for references. 
This literature tends to be more applied, although theoretical
results have been obtained \citep[e.g.][]{Li:Rohe:2015:1}. The two bodies of work are largely distinct,
with a few notable exceptions, such as the results on 
identifiability problems in \citep{Crane:Dempsey:2015:1}.

The specific problem considered here---the relationship between sampling and symmetry---seems largely unexplored,
but Aldous reasons about
exchangeability in terms of uniform sampling in \citep{Aldous:2010:1}, and, in joint work with Lyons 
\citep{Aldous:Lyons:2007:1}, extends the work of \citet{Benjamini:Schramm:2001:1} from a symmetry perspective.
(\citet{Kallenberg:1999} and other authors use the term \emph{sampling} differently, for
the explicit generation of a draw from a suitable representation of a distribution.)
More closely related from a technical perspective than experimental design in networks are two
ideas popular in combinatorics.
One is \emph{property testing}, which samples uniform parts of a large random structure $X$,
and then asks with what probability $X$ is in a certain class; see 
\citep{Alon:Spencer:2008}.
The second is used to define convergence of discrete structures:
Start with a set $\OUT$ of such structures, and equip it with some initial metric (so convergence
in distribution is defined).
Invent a randomized algorithm that generates a
``substructure'' $S(x)$ of a fixed structure ${x\in\OUT}$, and call a sequence ${(x_n)}$ convergent
if the laws ${\Law(S(x_n))}$ converge weakly.
The idea is exemplified by \citep{Benjamini:Schramm:2001:1}, but seems
to date back further, and is integral to the construction of graph limits:
The topology on dense graphs defined in this manner by 
\cref{alg:uvertex:simple} is metrizable, by the ``cut metric'' of 
Frieze and Kannan 
\citep{Borgs:Chayes:Lovasz:Sos:Vesztergombi:2008,Borgs:Chayes:Lovasz:Sos:Vesztergombi:2012:1}.

\newpage

\section{Spaces of graphs and discrete structures}
\label{sec:spaces}

Informally, we consider a space $\OUT$ of infinite structures, spaces $\OUT_n$ of finite
substructures of size $n$, and a map ${x\mapsto\textrestn{x}}$ that takes a structure of size $\geq n$
to its initial substructure of size $n$. For example, if $\OUT$ (resp.\ $\OUT_n$) consists of labeled graphs
vertex set $\mathbb{N}$ (resp.\ $[n]$), then ${\textrestn{\argdot}}$ may map to the induced subgraph on the
first $n$ vertices. More formally, these objects are defined as follows:
Let $\OUT_n$, for ${n\in\mathbb{N}}$, be countable sets. Require that for each pair ${m\leq n}$, there is a surjective map 
\begin{equation}
  \label{eq:restriction}
  \restn{\argdot}:\OUT_{n}\rightarrow\OUT_m
  \qquad\text{ such that }\qquad
  \rest[k]{\rest[m]{x_n}}=\rest[k]{x_n}\qquad\text{ if }x_n\in\OUT_n\text{ and }k\leq m\leq n\;.
\end{equation}
We write ${x_m\preceq x_n}$ whenever ${\textrest[m]{x_n}=x_m}$. In words, ${x_m}$ is a substructure of ${x_n}$.
An infinite sequence
\begin{equation}
  \label{eq:coherence}
  x:=(x_1,x_2,\ldots) 
  \quad\text{ with }\quad 
  x_n\in\OUT_n 
  \quad\text{ and }\quad 
  x_n\preceq x_{n+1}
  \quad\text{ for all }n\in\mathbb{N}\;
\end{equation}
can then be regarded as a structure of infinite size. The set of all such sequences is denoted $\OUT$.
The maps $\textrestn{\argdot}$
can be extended to $\OUT$ by defining ${\textrestn{x}=x_n}$ if ${x=(x_1,x_2,\ldots)}$ as above.

If each point in $\OUT$ is an infinite graph, two natural ways to measure size of subgraphs is
by counting vertices, or by counting edges. Since the notion of size determines the definition
of the restriction map ${x\mapsto\textrestn{x}}$, the two lead to different types of almost
discrete spaces:\\[-.5em]

{\noindent(i) \textit{Counting vertices.}}
Choose $\OUT$ as the set of graphs a given type (e.g.\ simple and undirected) with vertex set ${\bbN}$,
and ${\OUT_n}$ as the analogous set of graphs with vertex set ${[n]}$. The restriction map
${x\mapsto\textrestn{x}}$ extracts the induced subgraph on the first $n$ vertices, 
\ie $\textrestn{x}$ is the graph with vertex set $[n]$ that contains those edges ${(i,j)}$ of $x$
with ${i,j\leq n}$. Graph size is the cardinality of the vertex set.\\[-.5em]

{\noindent(ii) \textit{Counting edges.}}
A graph $x$ with vertices in $\bbN$ is represented as a sequence of edges,
${x=((i_1,j_1),(i_2,j_2),\ldots)}$, where ${i_k,j_k\in\mathbb{N}}$. 
Each set ${\OUT_n\subset(\bbN^2)^n}$ consists of all graphs with $n$ edges,
${x_n=((i_1,j_1),\ldots,(i_n,j_n))}$, and vertex set ${\braces{i_k,j_k|k\leq n}\subset\bbN}$.
The restriction map
\begin{equation}
  x\mapsto\restn{x}:=((i_1,j_1),\ldots,(i_n,j_n))
\end{equation}
extracts the first $n$ edges, and graphs size is cardinality of the edge set.
\\[-.5em]

To define probabilities on $\OUT$ requires a notion of measurability, and hence a topology.
We endow each countable set $\OUT_n$ with its canonical, discrete topology, and $\OUT$
with the smallest topology that makes all maps ${x\mapsto\textrestn{x}}$ continuous.
A topological space constructed in this manner is called \kword{procountable}. 
Any procountable space admits a canonical ``prefix metric''
\begin{equation}
  d(x,x'):=\inf_{n\in\mathbb{N}}\bigl\lbrace 2^{-n}\,\big\vert\,\textrestn{x}=\textrestn{x'}\bigr\rbrace\;,
\end{equation}
which is indeed an ultrametric. The ultrametric space $(\OUT,d)$ is complete.
An \kword{almost discrete space} is a proucountable space that is separable, and hence Polish.
Throughout, all sets of infinite graphs are almost discrete spaces, or subsets of such spaces.
If every set
$\OUT_n$ is finite, $\OUT$ is called \kword{profinite} (or \emph{Boolean}, or a \emph{Stone space})
\citep{Givant:Halmos:2009:1}. A space is profinite iff it is almost discrete and compact.
A random element $X$ of $\OUT$ is defined up to almost sure equivalence by a sequence ${X_1,X_2,\ldots}$
satisfying \eqref{eq:coherence} almost surely. A probability measure $P$
on $\OUT$ is similarly defined by its ``finite-dimensional distributions'' ${P_n=\textrestn{P}}$
on the spaces $\OUT_n$, by standard arguments \citep[e.g.][]{Kallenberg:2001}. 
This representation can be refined to represent a measures on topological
subspaces of $\OUT$; see \cref{app:addenda}.

For example, if $\OUT_n$ is the finite set of simple, undirected graphs with vertex set $[n]$,
and $\textrestn{\argdot}$ extracts the induced subgraph on the first $n$ vertices,
$\OUT$ is the space of all simple, undirected graph on $\mathbb{N}$, and its
topology coincides with the one inherited
from the product topology on $\braces{0,1}^{\infty}$.
This space $\OUT$ is the natural habitat of graph limit theory. The ``cut norm'' topology
\citep{Borgs:Chayes:Lovasz:Sos:Vesztergombi:2008} coarsens the almost discrete
topology on $\OUT$.

\section{Subsampling algorithms}
\label{sec:sampling}

To formalize the notion of subsampling, first consider a finite input graph 
${y_n\in\IN_m}$. The algorithm has access to a random 
source, which generates a sequence ${U_1,U_2,\ldots}$ of \iid uniform random variables in $[0,1]$, with
joint law $\PAlg$. Given $y_n$ and the
first $k$ uniform variables, the algorithm generates a random output graph
$\Alg{n}{k}(y_n,U_1,\ldots,U_k)$ in $\OUT_k$. Formally, this is a (jointly) measurable map
\begin{equation}
  \label{eq:def:sampler:1}
  \Alg{n}{k}:\IN_n\times[0,1]^k\rightarrow\OUT_k\;,
\end{equation}
which we will often read as an ${\OUT_k}$-valued random variable $\Alg{n}{k}(y_n)$, parametrized by $y_n$. 
Since each sampling step augments the previously sampled graph, we require these random variables
to cohere accordingly, as
\begin{equation}
  \label{eq:def:sampler:2}
  \Alg{n}{k}(y_n)\preceq\Alg{n}{k+1}(y_n)\qquad\text{almost surely.}
\end{equation}
It suffices to require that ${\Alg{n}{k}(y)}$ exists for $n$ sufficiently large: For each
$y$ and $k$, there is an ${n(k)}$ such that $\Alg{n}{k}(y)$ is defined for all ${n\geq n(k)}$.
A \kword{sampling algorithm} is a family ${(\Alg{n}{k})_{k\in\mathbb{N},n\geq n(k)}}$
as in \eqref{eq:def:sampler:1} that satisfies \eqref{eq:def:sampler:2}.

To explain sampling from an infinite graph ${y\in\IN}$, 
we ask whether there is a limiting variable $\AlgInf{k}(y)$ such 
that convergence ${\Alg{n}{k}(\textrestn{y})\rightarrow\AlgInf{k}(y)}$ 
holds in distribution as ${n\rightarrow\infty}$.
For that to be the case, it is certainly necessary that the limits
\begin{equation*}
  \hom_{x_k}(y):=\lim_{n\rightarrow\infty} \PAlg\bigbraces{\Alg{n}{k}(\textrestn{y})=\textrestk{x}}
\end{equation*}
exist for all ${x\in\OUT,k\in\mathbb{N}}$.
We call the limit $\hom_{x_k}(y)$ a \kword{prefix density}. These 
prefix densities can be collected into a vector,
\begin{equation*}
  \label{eq:prefix:vector}
  \mathbf{\hom}(y):=\bigl(\hom_{x_k}(y)\bigr)_{k\in\mathbb{N},x_k\in\OUT_k}
  \qquad\text{ which is a measurable map }\qquad
  \mathbf{\hom}:\subIN\rightarrow[0,1]^{\cup_k\OUT_k}\;.
\end{equation*}
Our first main result shows the necessary condition that $\mathbf{\hom}$ exists is indeed sufficient:
\begin{theorem}
  \label{result:alg:limit}
  Let ${\IN}$ be an almost discrete space, and $\subIN$ any  subset,
  equipped with the restriction $\borel(\subIN)$ of the Borel sets of $\IN$.
  Let ${S=(\Alg{n}{k})}$ be a sampling algorithm ${\subIN\rightarrow\OUT}$.
  If the prefix densities $\mathbf{\hom}(y)$ exist on 
  ${\subIN}$, there exists a jointly measurable function
  \begin{equation*}
    \label{eq:alg:limit}
    \AlgInf{\infty}:\subIN\times[0,1]\rightarrow\OUT
    \qquad\text{ satisfying }\qquad
    \Alg{n}{k}(\textrestn{y},U)\;\;\xrightarrow{\;\;\text{\rm\tiny d}\;\;}\;\;\restk{\AlgInf{\infty}(y,U)}
    \qquad\text{ as }n\rightarrow\infty
  \end{equation*}
  for all ${y\in\subIN}$ and ${k\in\mathbb{N}}$.
\end{theorem}
There is hence a random variable $\AlgInf{\infty}(y)$,
with values in $\OUT$, which can be interpreted as an infinite or ``asymptotically large'' sample from
an infinite input graph $y$. Each restriction ${\AlgInf{k}(y)=\textrestk{\AlgInf{\infty}(y)}}$
represents a sample of size $k$ from $y$.
If repeated application of sampling preserves the output distribution, \ie if
\begin{equation*}
  \Alg{m}{k}(\AlgInf{m}(y))\equdist\AlgInf{k}(y)
  \quad\text{ and }\quad
  \Alg{m}{k}(\Alg{n}{m}(y))\equdist\Alg{n}{k}(y)
  \qquad\text{ whenever } k\leq m\leq n\;,
\end{equation*}
we call the algorithm \kword{idempotent}.

\begin{remark}
  \label{remark:population}
  The limit in output size $k$ is an inverse limit: A growing graph is assembled as in \eqref{eq:coherence},
  and all information in $\AlgInf{k}$ can be recovered from $\AlgInf{\infty}$.
  In contrast, the limit in input size is distributional, so the input graph $y$ can typically
  \emph{not} be reconstructed, even from an infinite sample $\AlgInf{\infty}(y)$. 
  The limit of \cref{alg:uvertex:simple}, for example, will output an empty graph if $y$ has a finite
  number of edges, or indeed if the number of edges in $\textrestn{y}$ grows
  sub-quadratically in $n$. We regard $S$ as a measurement of properties of 
  a ``population'' (see \citep{Kolaczyk:2009:1} for a discussion of populations in network problems,
  and \citep{Orbanz:Roy:2014} for graph limits as populations underlying
  exchangeable graph data). An
  infinitely large sample $\AlgInf{\infty}(y)$ makes asymptotic statements valid, 
  in the sense that any effect of finite sample
  size can be made arbitrarily small, but does not exhaust the population.
\end{remark}

\section{Background: Invariance and symmetry}
\label{sec:symmetry}

We use the term \emph{invariance} to describe preservation of probability distributions under
a family of transformations; we also call an invariance a \emph{symmetry} 
if this family is specifically a group.
Let $\OUT$ be a standard Borel space, and $\fam$ a family of measurable (but not necessarily invertible)
transformations ${\OUT\rightarrow\OUT}$. A random element $X$ of $\OUT$ is 
\kword{$\fam$-invariant} if its law remains invariant under every element of $\fam$,
\begin{equation}
  t(X)\quad\equdist\quad X \qquad\text{ for all }t\in\fam\;.
\end{equation}
Analogously, a probability measure $P$ is $\fam$-invariant if the image measure ${t(P)=P\circ t^{-1}}$
satisfies ${t(P)=P}$ for all ${t\in\fam}$. 
We denote the set of all $\fam$-invariant probability
measures on $\OUT$ by ${\INV(\fam)}$. It is trivially
convex, though possibly empty if the family $\fam$ is ``too large''.

\subsection{Ergodicity}

Inference from a single instance relies on the concept of ergodicity.
A Borel set ${A\subset\OUT}$ is \kword{invariant} if ${t(A)=A}$ for all ${t\in\fam}$, and
\kword{almost invariant} if
\begin{equation*}
  P(A\vartriangle t^{-1}A)=0\qquad\text{ for all }t\in\fam\text{ and all }P\in\INV\;.
\end{equation*}
We denote the system of all invariant sets $\sigma(\fam)$, and that of all almost invariant
sets ${\bsigma(\fam)}$. Both are $\sigma$-algebras.
Recall that a probability $P$ is \kword{trivial} on a $\sigma$-algebra $\sigma$ if 
${P(A)\in\braces{0,1}}$ for all ${A\in\sigma}$.
For a probability measure $P$ on $\OUT$, we define:
\begin{equation*}
  P \text{ is \kword{$\fam$-ergodic}}
  \quad:\Leftrightarrow\quad
  P \text{ is $\fam$-invariant, and trivial on }\bsigma(\fam)\;.
\end{equation*}
The set of all ergodic measures is denoted ${\ERG(\fam)}$.

\subsection{Groups and symmetries}

We reserve the term \emph{symmetry} for invariance under transformation families $\fam$ that
form a group. 
A useful concept in this context is the notion of a \emph{group action}: For example, if $\OUT$ is a space
of graphs, a group of permutations may act on a graph ${x\in\OUT}$ by permuting
its vertices, by permuting its edges, by permuting certain subgraphs, etc. Such different
effects of one and the same group can be formalized as maps $T(\phi,x)$ that explain
how a permutation $\phi$ affects the graph $x$. Formally, let
$\group$ be a group, with unit element $e$.
An \kword{action} of $\group$ on $\OUT$ is a map ${T:\group\times\OUT\rightarrow\OUT}$,
customarily denoted ${T_{\phi}(x)=T(\phi,x)}$, with the properties
\begin{equation*}
  \text{(i)}\quad T_e(x)=x \quad\text{ for all }x\in\OUT 
  \qquad\text{ and }\qquad
  \text{(ii)}\quad T_{\phi}\circ T_{\phi'}=T_{\phi\phi'}\quad\text{ for all }\phi,\phi'\in\group\;.
\end{equation*}
If $\group$ is equipped with a topology, and with the corresponding Borel sets, $T$ is a
\kword{measurable action} if it is jointly measurable in both arguments.
Any measurable action $T$ on $\OUT$ defines a family ${\fam:=T(\group):=\braces{T_{\phi}|\phi\in\group}}$
of transformations ${\OUT\rightarrow\OUT}$. Clearly, each element of $\fam$ is a bimeasurable bijection, 
and $\fam$ is again a group. 
The \kword{orbit} of an element $x$ of $\OUT$ under a group action is the set
${T_{\group}(x):=\braces{T_\phi(x), \phi\in\group}}$.
The orbits of $T$ form a partition of $\OUT$ into disjoint sets. 
If there exists a Polish topology on $\group$ that makes $T$ measurable, each orbit
is a measurable set \citep[][\S 2.3]{Becker:Kechris:1996}. 

\subsection{Characterization of ergodic components}
The main relevance of ergodicity to our purposes is that, informally,
the elements of a model $\mathcal{P}$ chosen as a subset 
  ${\mathcal{P}\subset\ERG(\fam)}$ can be distinguished from one another by means of a single
  realization, 
provided that $\fam$ is not too complex.
In other words, if a random variable $X$ is assumed to be distributed according to \emph{some} distribution
in $\mathcal{P}$, then a single draw from $X$ determines this distribution unambiguously within $\mathcal{P}$.
That is a consequence of the ergodic decomposition theorem, whose various shapes and guises are part of  
mathematical folklore. 
To give a reasonably general statement, we have to 
formalize that $\fam$ be ``not too complex'':
Call $\fam$ \kword{separable} if a countable subset ${\fam_0\subset\fam}$ exists
that defines the same set of invariant measures as $\fam$,
\begin{equation}
  \INV(\fam_0)=\INV(\fam)\;.
\end{equation}
If so, we call $\fam_0$ a \kword{separating subset}. Criteria for verifying separability
are reviewed in \cref{app:addenda}. 
If $\fam_0$ is separating,
both the ergodic probability measures and the almost invariant sets defined by the 
two families coincide,
\begin{equation}
  \ERG(\fam_0)=\ERG(\fam) \qquad\text{ and }\qquad \bsigma(\fam_0)=\bsigma(\fam)\;.
\end{equation}
The following form of the decomposition theorem is amalgamated from 
\citep{Farrell:1962:1,Greschonig:Schmidt:2000:1,Maitra:1977}:
\begin{theorem}[Folklore]
  \label{theorem:ergodic:decomposition}
  Let $\fam$ be a separable family of measurable transformations of a standard Borel space $\OUT$.
  Then the $\fam$-ergodic measures are precisely the extreme points of $\INV(\fam)$, and for every
  pair ${P\neq P'}$ of ergodic measures, there is a set ${A\in\bsigma(\fam)}$ such that
  ${P(A)=1}$ and ${P'(A)=0}$.
  A random element $X$ of $\OUT$ is $\fam$-invariant if and only if there is a random probability
  measure $\xi$ on $\OUT$ such that
  \begin{equation}
    \label{eq:ergodic:decomposition}
    \xi\in\ERG(\fam) 
    \quad\text{ and }\qquad
    P[X\in\argdot|\xi]=\xi(\argdot)
  \end{equation}
  almost surely. If so, the law of $\xi$ is uniquely determined by the law of $X$.
\end{theorem}
The decomposition theorem can be understood as a generalization of the representation
theorems of de Finetti, of Aldous and Hoover, and similar results: In expectation, the almost
sure identity \eqref{eq:ergodic:decomposition} takes the weaker but more familiar form
\begin{equation*}
  P(\argdot)=\int_{\ERG(\fam)}\nu(\argdot)\mu_{\xi}(d\nu)\qquad\text{ where }\mu_{\xi}:=\Law(\xi)\;.
\end{equation*}
In de Finetti's theorem, the ergodic measures are the laws of \iid sequences; in the
Aldous-Hoover theorem applied to simple, undirected, exchangeable graphs, they are those distributions
represented by graphons; etc.
The generality of \cref{theorem:ergodic:decomposition} comes at a price:
The theorems of de Finetti, Kingman, and Aldous-Hoover provide
constructive representations of \eqref{eq:ergodic:decomposition}: There is a collection $(U_i)$ of independent,
uniform random variables on $[0,1]$, and a class of measurable mappings $\mathcal{H}$, such that
each ergodic random element $X$ can be represented as ${X\equdist h\bigl((U_i)\bigr)}$ for
some ${h\in\mathcal{H}}$. The representation is non-trivial in that each finite substructure
$\textrestk{X}$ can be represented analogously by a finite subset of the collection
${(U_i)}$. \citet{Kallenberg:2005} calls such a representation a \kword{coding}.
Existence of a coding can be much harder to establish than \eqref{eq:ergodic:decomposition},
and not all invariances seem to admit codings.

\subsection{Definitions of exchangeability}

The term \emph{exchangeability} generically refers to invariance under an action of either
the finitary symmetric group $\FSym$, or of the infinite symmetric group $\Sym{\mathbb{N}}$
of all bijections of $\mathbb{N}$. For the purposes of \cref{theorem:ergodic:decomposition},
both definitions are typically equivalent:
The group $\Sym{\mathbb{N}}$ inherits its natural topology from the product space
${\mathbb{N}^{\mathbb{N}}}$, which makes the subgroup
${\FSym}$ a dense subset. 
If $T$ is a continuous action of $\Sym{\mathbb{N}}$
on a metrizable space, the image $T(\FSym)$ hence lies dense
in $T(\Sym{\mathbb{N}})$ in pointwise convergence, which in turn implies
$T(\FSym)$ is a separating subset for $T(\Sym{\mathbb{N}})$ (see \cref{sec:symmetry:addenda}).

In terms of their orbits, the two actions differ drastically: 
If ${\OUT=\braces{0,1}^{\infty}}$, for example,
and $T$ permutes sequence indices, each orbit of $\FSym$ is countable.
Not so for $\Sym{\mathbb{N}}$: Let ${A\subset\OUT}$ be the set of sequences
containing both an infinite number of 0s and of 1s. For any two
${x,x'\in A}$, there exists a bijection ${\phi\in\Sym{\mathbb{N}}}$ with ${x'=T_{\phi}(x)}$. 
The set $A$ thus constitutes a single, uncountable orbit of $\Sym{\mathbb{N}}$, which is complemented by a countable
number of countable orbits.
That illustrates the role of almost invariant sets: By de Finetti's theorem, the ergodic
measures are factorial Bernoulli laws. For all Bernoulli parameters ${p\in (0,1)}$, 
these concentrate on $A$, and $A$ does not subdivide further into strictly
invariant sets. In other words, $\sigma(\Sym{\mathbb{N}})$ does not provide sufficient resolution
to guarantee
mutual singularity in \cref{theorem:ergodic:decomposition}, but the \emph{almost} invariant sets 
${\bsigma(\Sym{\mathbb{N}})}$ do. \citet{Vershik:2004:1} gives a detailed account.
Unlike \cref{theorem:ergodic:decomposition}, more explicit results like the law of large numbers in 
\cref{sec:lln} rely on the orbit structure, and must be formulated in terms of $\FSym$.

\section{Sampling and symmetry}
\label{sec:sampling:symmetry}

We now consider the fundamental problem of drawing conclusions from a single observed instance
in the context of sampling.
For now, we assume the entire, infinite output graph $\AlgInf{\infty}(y)$ is available. Consider
a sampling algorithm ${S}$, with input set ${\subIN\subset\IN}$ and output space ${\OUT}$, 
defined as in \cref{sec:sampling},
whose prefix densities $\mathbf{\hom}(y)$ exist for all ${y\in\subIN}$. We generically denote its
output distributions
\begin{equation*}
  P_y:=\Law(\AlgInf{\infty}(y))\;.
\end{equation*}
Suppose a model $\mathcal{P}$ is chosen as a subset of ${\braces{P_y|y\in\subIN}}$.
Can two elements ${P_y,P_{y'}\in\mathcal{P}}$ be distinguished from another given a single sample 
$\AlgInf{\infty}(y)$? That can be guaranteed only if
\begin{equation}
  \label{eq:sigma:separation}
  P_y(A)=1 \quad\text{ and }\quad P_{y'}(A)=0\qquad\text{ for some Borel set }A\subset\OUT\;.
\end{equation}
To decide more generally which distribution in ${\braces{P_y|y\in\subIN}}$ (and hence which
input graph $y$) accounts for $\AlgInf{\infty}(y)$, we define
\begin{equation*}
  \Sigma:=\bigcap_{y\in\subIN}\Sigma_y
  \qquad\text{ where }\qquad
  \Sigma_y:=\bigl\{A\in\borel(\OUT)\,\big\vert\,P_y(A)\in\braces{0,1}\bigr\}\;.
\end{equation*}
Then $\Sigma$ is a $\sigma$-algebra. From \eqref{eq:sigma:separation}, we conclude:
\begin{center}
  \emph{Determining the input graph based on a single realization of $\AlgInf{\infty}$ is 
    possible if the output laws $P_y$ are pairwise distinct on $\Sigma$.}
\end{center}
The sampling algorithm does not typically preserve all information provided by the input graph,
due to the distributional limit defining $\AlgInf{\infty}$. Thus, demanding that all pairs
${y\neq y'}$ be distinguishable may be too strong a requirement. 
The $\sigma$-algebra $\Sigma$ defines a natural equivalence relation on $\subIN$,
\begin{equation*}
  y\equivIN y'
  \quad:\Leftrightarrow\quad
  P_y(A)=P_{y'}(A)\text{ for all }A\in\Sigma\;.
\end{equation*}
More colloquially, ${y\equivIN y'}$ means $y$ and $y'$ 
cannot be distinguished given a single realization ${\AlgInf{\infty}(y)}$.
We note ${y\equivIN y'}$ does not generally imply ${P_y=P_{y'}}$: The measures
may be distinct, but detecting that difference may require multiple realizations. We call
the algorithm \kword{resolvent} if
\begin{equation}
  y\equivIN y' \quad\text{ implies }\quad P_y=P_{y'}\;.
\end{equation}
Let $\hat{y}$ denote the equivalence class of $y$. If $S$ is resolvent, we can 
define ${P_{\hat{y}}:=P_y}$, and formulate the condition above as
\begin{equation}
  \label{eq:sigma:separation:hat}
  P_{\hat{y}} \text{ and }P_{\hat{y}'}\text{ are mutually singular on }
  \Sigma \text{ whenever }\hat{y}\neq\hat{y}'\;.
\end{equation}
Establishing mutual singularity requires identifying
a suitable system of sets $A$ in \eqref{eq:sigma:separation}, which can be all but impossible:
Since $\OUT$ is Polish, each measure $P_y$ has a unique support
(a smallest, closed set $F$ with ${P_y(F)=1}$), but these closed support sets are \emph{not} generally
disjoint. To satisfy \eqref{eq:sigma:separation}, $A$ is chosen more generally as measurable,
but unlike the closed support, the measurable support of a measure is far from unique. One would hence
have to identify a (possibly uncountable) system of not uniquely determined sets, each chosen
just so that \eqref{eq:sigma:separation} holds pairwise.

If an invariance holds, \cref{theorem:ergodic:decomposition} solves the problem.
That motivates the following definition:
A measurable action $T$ of a group $\group$ on $\OUT$ is a \kword{symmetry} of the
algorithm $S$
if all output distributions $P_y$ are $\group$-ergodic. If $\group$ is countable, that is equivalent to demanding
\begin{equation}
  \text{(i)} \quad
  T_{\phi}(\AlgInf{\infty}(y))\equdist \AlgInf{\infty}(y)\quad
  \text{ for all }y\in\subIN,\phi\in\group
  \qquad\text{ and }\qquad
  \text{(ii)} \quad
  \sigma(\group)\subset\Sigma\;.
\end{equation}
If $\group$ is uncountable, (ii) must be strengthened to ${\bsigma(\group)\subset\Sigma}$.
Clearly, an algorithm that admits a separable symmetry is resolvent; thus, symmetry 
guarantees \eqref{eq:sigma:separation:hat}.
We note mutual singularity
could be deduced without requiring $\fam$ is a group action; this condition
anticipates the law of large numbers in \cref{sec:lln}.

\subsection{A remark: What can be said without symmetry}
\label{sec:no:symmetry}

If we randomize the input graph by substituting a random element $Y$ of $\subIN$ with law $\nu$ for $y$,
the resulting output distribution is the mixture ${\int P_y\nu(dy)}$. We define the set of all such laws as
${M:=\braces{\int_{\subIN}P_y\nu(dy)\,\vert\,\nu\in\pMeas(\subIN)}}$, where $\pMeas(\subIN)$ is the
space of probability measures on $\subIN$.
Clearly, $M$ is convex, with the laws $P_y$ as its extreme points.
Without further assumptions, we can obtain the following result.
It makes no appeal to invariance, and cannot be deduced from \cref{theorem:ergodic:decomposition} above.
\begin{proposition}
  \label{result:quasi:invariance}
  Let $S$ be a sampling algorithm with prefix densities. Then for every ${P\in M}$, there exists
  a measurable subset ${Q_P\subset\pMeas(\OUT)}$ of probability measures such that all measures
  in $Q_P$ are (i) mutually singular and (ii) 0--1 on $\Sigma$. There exists a random probability
  measure $\xi_P$ on $\OUT$ such that 
  ${\xi_P\in Q_P}$ and ${P[X\in\argdot|\xi_P]=\xi_P}$ almost surely.
\end{proposition}
A structure similar to \cref{theorem:ergodic:decomposition} is clearly recognizable.
That said, the result is too weak for our purposes: The set of $Q_P$ of representing
measures depends on $P$, which means it cannot be used as a model, and 
the result does not establish a relationship between the measure $P_y$ and the elements of $Q_P$. Note
it holds if, but not only if, ${P\in M}$.

\section{Symmetric laws of large numbers}
\label{sec:lln}

Consider a similar setup as above: A random variable $X$ takes values in a standard
Borel space $\OUT$, and its 
distribution $P$ is invariant under a measurable action $T$ of a group $\group$. 
Let ${f:\OUT\rightarrow\mathbb{R}}$ be a function in ${\Lone(X)}$.
If $T$ is separable, \cref{theorem:ergodic:decomposition} 
shows that $X$ is generated by drawing an instance of $\xi$---that
is, by randomly selecting an ergodic measure---and then drawing ${X|\xi\sim\xi}$. The expectation
of $f$ given the instance of $\xi$ that generated $X$ is
\begin{equation*}
  \xi(f)
  \quad=\quad
  \mean[f(X)|\xi]
  \quad=\quad
  \mean[f(x)|\bsigma(\group)]\qquad\text{a.s.}
\end{equation*}
Again by \cref{theorem:ergodic:decomposition}, 
observing $X$ completely determines the instance of $\xi$. In principle, 
$X$ hence completely determines ${\mean[f(X)|\bsigma(\group)]}$. These are all abstract
quantities, however; is it possible to compute ${\mean[f(X)|\bsigma(\group)]}$ from a
given instance of $X$?

If the group is finite, the elementary properties of conditional expectations imply
\begin{equation*}
  \mean[f(X)|\bsigma(\group)] 
  \quad=\quad
  \frac{1}{|\group|}\sum_{\phi\in\group}f(T_{\phi}(X))\qquad\text{ almost surely,}
\end{equation*}
so $\xi(f)$ is indeed given explicitly. The groups arising in the context of sampling are
typically countably infinite. In this case, the average
on the right is no longer defined. It is then natural to ask whether
$\xi(f)$ can be approximated by finite averages, \ie whether there
are finite sets ${\A_1,\A_2,\ldots\subset\group}$ such that 
\begin{equation*}
  \frac{1}{|\A_k|}\sum_{\phi\in\A_k}f(T_{\phi}(X))
  \quad\xrightarrow{n\rightarrow\infty}\quad
  \mean[f(X)|\bsigma(\group)] 
  \qquad\text{ almost surely.}
\end{equation*}
Since ${\mean[f(X)|\bsigma(\group)]}$ is invariant under each ${\phi\in\group}$,
each average on the left must be invariant at least approximately:
A necessary
condition for convergence is certainly that, for any ${\phi\in\group}$, the 
relative size of the displacement ${\phi\A_k\!\vartriangle\!\A_k}$ can be made
arbitrarily small by choosing $k$ large. That is formalized in the next condition,
\eqref{eq:Folner}(i).

A countable group is \kword{amenable} if there is a sequence 
${\A_1,\A_2,\ldots}$ of finite subsets of $\group$ with the property:
For some ${c>0}$ and all ${\phi\in\group}$,
\begin{equation}
  \label{eq:Folner}
  \text{(i)}\quad
  |\phi\A_k \cap \A_k|\xrightarrow{n\rightarrow\infty}|\A_k|
  \quad\text{ and }\quad
  \text{(ii)}\quad
  \bigl|\cup_{j<k}\A_j^{-1}\A_k\bigr|\leq c|\A_k| \text{ for all }k\in\mathbb{N}.
\end{equation}
A sequence $(\A_k)$ satisfying (i) is called \kword{almost invariant}. This first condition
turns out to be the crucial one: If a sequence satisfying (i) exists, it
is always possible to find a sequence satisfying (i) and (ii), by passing
to a suitable subsequence if necessary \citep[][Proposition 1.4]{Lindenstrauss:2001:1}.
Thus, $\group$ is a amenable if it contains a sequence satisfying (i). 
Amenable groups arise first and foremost in ergodic theory \citep[e.g.][]{Einsiedler:Ward:2011:1},
but also, for example, in hypothesis testing, as the natural class of groups satisfying the
Hunt-Stein theorem \citep{Bondar:Milnes:1981:1}. 
If \eqref{eq:Folner} holds, and $T$ is a measurable action of $\group$ on $\OUT$,
we call the measurable mapping
\begin{equation}
  (x,k)\mapsto \empavg_k^x(\argdot):=\frac{1}{|\A_k|}\sum_{\phi\in\A_k}\delta_{T_{\phi}(x)}(\argdot)
\end{equation}
an \kword{empirical measure} for the action $T$.
\begin{theorem}
  \label{lemma:main}
  Let $X$ be a random element of a Polish space $\OUT$, and ${f,f_1,f_2,\ldots}$ functions
  ${\OUT\rightarrow\mathbb{R}}$ in $\Lone(X)$, such that ${f_k\rightarrow f}$ almost surely
  under the law of $X$.
  Let $T$ be a measurable action of a countable group satisfying \eqref{eq:Folner}, and $\empavg$ the empirical
  measure defined by $(\A_k)$.
  If $X$ is invariant under $T$, then
  \begin{equation}
    \label{eq:lln}
    \empavg_k^{X}(f_k)\xrightarrow{n\rightarrow\infty}\xi(f)
    \qquad \text{almost surely,}
  \end{equation}
  where $\xi$ is the random ergodic measure in \eqref{eq:ergodic:decomposition}.
  If moreover there is a function ${g\in\Lone(X)}$ such that ${|f_k|\leq g}$ for all $k$,
  convergence in \eqref{eq:lln} also holds in $\Lone(X)$.
\end{theorem}
The finitary symmetric group $\FSym$ satisfies \eqref{eq:Folner} for ${\A_k:=\Sym{k}}$. 
The law of large numbers \eqref{eq:lln} hence holds generically for any ``exchangeable random structure'',
\ie for any measurable action of $\FSym$. 
Special cases include the law of large numbers for de Finetti's theorem, the continuity of Kingman's 
correspondence \citep[][Theorem 2.3]{Pitman:2006}, and 
Kallenberg's law of large numbers for exchangeable arrays \citep{Kallenberg:1999}.
They can be summarized as follows:
\begin{corollary}
  If a random element $X$ of a Polish space is invariant under a measurable action $T$ of $\FSym$, the
  empirical measure
  ${\frac{1}{k!}\sum_{\phi\in\Sym{k}}\delta_{T_{\phi}(X)}}$ converges weakly to $\xi$ as ${k\rightarrow\infty}$,
  almost surely under the law of $\xi$.
\end{corollary}
For sequences, the empirical measure can be broken down further into a sum over sequence entries,
and redundancy of permutations then shrinks the sum from $k!$ to $k$ terms.
Now suppose that $X$ is specifically the output $\AlgInf{\infty}$ of a sampling algorithm:
\begin{corollary}
  \label{corollary:lln:sampler}
  Let ${S:\subIN\rightarrow\OUT}$ be a sampling algorithm whose prefix densities exist for all ${y\in\subIN}$.
  Suppose a countable amenable group $\group$ is a symmetry group of $S$ under a measurable action $T$. If $S$
  samples from a random input graph $Y$, then
  \begin{equation*}
    \empavg_k^{\AlgInf{\infty}(Y)}(f_k)\xrightarrow{k\rightarrow\infty}P_Y(f)\qquad\Law(Y)\text{-a.s.}
  \end{equation*}
  holds for any functions ${f,f_1,f_2,\ldots}$ satisfying ${f\in\Lone(P_y)}$ and ${(f_k)\rightarrow f}$ $P_y$-a.s.\ for
  $\Law(Y)$-almost all $y$.
\end{corollary}
For example, one can fix a finite structure $x_j$ of size $j$, and choose $f$ as the indicator 
${f(x):=\mathbb{I}\braces{\textrest[j]{x}=x_j}}$. \cref{corollary:lln:sampler} then implies
\begin{equation*}
  \frac{1}{\A_k}\sum_{\phi\in\A_k}\mathbb{I}\braces{\textrest[j]{T_{\phi}(\AlgInf{\infty}(y))}=x_j}
  \quad\xrightarrow{k\rightarrow\infty}\quad
  \hom_{x_j}(y)\;,
\end{equation*}
which makes ${\empavg_k^{\AlgInf{\infty}(y)}(\mathbb{I}\braces{\argdot=x_j})}$ 
a (strongly) consistent estimator of the prefix density $\hom_{x_j}$ from output generated by the
sampler. Here, $\AlgInf{\infty}(y)$ is still an infinite structure. If the action $T$ is such that the
elements of each set $\A_k$ affect only the initial substructure of size $k$,
we can instead define 
${f_k(x_k):=\mathbb{I}\braces{\textrest[j]{x_k}=x_j}}$ for graphs $x_k$ of size ${k\geq j}$.
Thus, ${f_k:\OUT_k\rightarrow\braces{0,1}}$, and ${f_k(\textrestk{x})=f(x)}$. If a sample 
${\AlgInf{1}(y)\preceq\AlgInf{2}(y)\preceq\ldots}$ is generated from $y$ using $S$,
\begin{equation*}
  \frac{1}{\A_k}\sum_{\phi\in\A_k}\mathbb{I}\braces{\textrest[j]{T_{\phi}(\AlgInf{k}(y))}=x_j}
  \quad\xrightarrow{k\rightarrow\infty}\quad
  \hom_{x_j}(y)
\end{equation*}
consistently estimates $\hom_{x_k}(y)$ from a finite sample of increasing size.
The sampling algorithms discussed in the next section admit such estimators.

\section{Sampling by random transformation}
\label{sec:transformation}

We now consider group actions where each element $\phi$ of the group $\group$ changes only a finite
substructure: ${T_{\phi}(x)}$ replaces a prefix $\textrestk{x}$ of $x$ by some other structure of size $n$.
We can hence subdivide the group into subsets $\group_n$, for each $n$, consisting of elements which only
affect the prefix of size $n$. Thus, ${\group_n\subset\group_{n+1}}$. If $\phi$ only affects a prefix of
size ${\leq n}$, then typically so does its inverse, and each subset $\group_n$ is itself a group.
If each subgroup $\group_n$ is finite, the group $\group$ is hence of the form
\begin{equation}
  \label{direct:union}
  \group=\cup_{n\in\mathbb{N}}\group_n \qquad\text{ for some finite groups }\group_1\subset\group_2\subset\ldots\;.
\end{equation}
A group satisfying \eqref{direct:union} is called \emph{direct limit} or \emph{direct union} of finite groups.
Since it is countable, any measurable action satisfies
\cref{theorem:ergodic:decomposition}. Plainly, $\group$ also satisfies
\eqref{eq:Folner}, with ${\A_n=\group_n}$. Thus, for any measurable action $T$,
\begin{equation*}
  (x,n)\mapsto \empavg^x_n(\argdot)=\sum_{\phi\in\group_n}\delta_{T_\phi(x)}(\argdot)
\end{equation*}
is an empirical measure, and satisfies the law of large numbers \eqref{eq:lln}.

If each $\phi$ affects only a finite substructure, the action must commute with restriction, in the sense that
\begin{equation}
  \label{eq:affects:only:finite:prefix}
  T_n(\phi,\textrestn{x})=\restn{T(\phi,x)} \text{ for an action } 
  T_n:\group_n\times\OUT_n\rightarrow\OUT_n
  \text{ and all }\phi\in\group_n, x\in\OUT\;.
\end{equation}
We call any action of a direct limit group that satisfies \eqref{eq:affects:only:finite:prefix}
a \kword{prefix action}.
In most cases, one can think of a prefix action $T_{\phi}$ as a map that removes the subgraph $\textrestn{x}$
from $x$ by some form of ``surgery'', and then pastes in another graph ${T_n(\phi,\textrestn{x})\in\OUT_n}$
of the same size. The action $T_n$ is hence a subset of the group ${\Sym{\OUT_n}}$ of all permutations
of $\OUT_n$. If $\OUT_n$ is finite, so is ${\Sym{\OUT_n}}$, which is hence a valid choice for $\group_n$.
Prefix actions include, for example, the case where $T_n$ is the action of $\Sym{n}$ on the first $n$
vertices, but it is worth noting that $\group_n$ can be much larger: ${\Sym{\OUT_n}}$ is typically
of size exponential in $\Sym{n}$. We observe:
\begin{proposition}
  \label{lemma:direct:union:continuous:action}
  Prefix actions on almost discrete spaces are continuous.
\end{proposition}

\subsection{Random transformations}
Transformation invariance can be built into a sampling
algorithm by constructing the algorithm from a random transformation. For a random element
$\Phi_n$ of $\group_n$, we define
\begin{equation}
  \label{eq:random:transformation:sampling}
  \Alg{n}{k}(y):=\restk{T(\Phi_n,y)}\qquad\text{ for each }y\in\subIN\subset\OUT\;.
\end{equation}
If $T$ is prefix action, one can equivalently substitute $\textrestn{y}$ for $y$ on the right-hand side.
\cref{alg:uvertex:simple} can for instance be represented in this manner, by choosing $\Phi_n$ as
a uniform random permutation of the first $m$ vertices.
The next results assume the following conditions:
\makebox[\textwidth][c]{
\begin{tikzpicture}[mybraces]
  \node at (-1,0) {\begin{minipage}[b]{.82\textwidth}
      \begin{tabular}{rl}
        (i) & $T$ is a prefix action of a direct limit $\group$ on an almost discrete space $\OUT$.\\[.3em]
        (ii) & The sampling algorithm $S$ is defined by \eqref{eq:random:transformation:sampling}, where
        each $\Phi_n$ is\\ & uniformly distributed on the finite group $\group_n$.\\[.3em]
        (iii) & Its prefix densities $\mathbf{\hom}$ 
      exist for all $y$ in a $T(\group)$-invariant subset ${\subIN\subset\OUT}$.
      \end{tabular}
  \end{minipage}};
  \draw[brace] (6,.9)--(6,-.9);
  \node at (7,0) {\eqnum\label{conditions:prefix:sampler}};
\end{tikzpicture}
}
The uniform random elements $\Phi_n$ used in the construction are only defined on finite subgroups,
but whenever prefix densities exist, one can once again take the limit in input size and obtain
a limiting sampler $\AlgInf{\infty}$. These samplers are particularly well-behaved:
\begin{theorem}
  \label{result:direct:union:sampler}
  Let $S$ be a sampling algorithm satisfying \eqref{conditions:prefix:sampler}. 
  Then for all ${\phi\in\group}$,
  \begin{equation*}
    \text{\rm (i)}\quad
    \mathbf{\hom}\circ T_{\phi}=\mathbf{\hom}\quad\text{for }\phi\in\group
    \qquad
    \text{\rm (ii)}\quad
    \mathbf{\hom}(\AlgInf{\infty}(y))\equas\mathbf{\hom}(y)
    \qquad
    \text{\rm (iii)}\quad
    \mathbf{\hom}(y)=\mathbf{\hom}(y')\;\text{ iff }\;y\equivIN y'\;.
  \end{equation*}
  Each output distribution ${P_y}$ is $\group$-invariant, and the law of a sample ${\AlgInf{k}(y)}$
  of size $k$ is ${\group_k}$-invariant.
  The algorithm is idempotent and resolvent, and any two output distributions ${P_y}$ and ${P_{y'}}$ are either 
  identical, or mutually singular.
\end{theorem}

One can ask whether it is even possible to recover properties of the input
graph: If ${\subIN\subset\OUT}$ and ${f:\OUT\rightarrow\mathbb{R}}$ is a statistic, can 
${f(y)}$ be estimated based on ${\AlgInf{\infty}(y)}$? 
Since the sampling algorithm does not resolve differences between to equivalent input graphs ${y\equivIN y'}$, 
a minimal requirement is that $f$ be constant on equivalence classes, 
\begin{equation}
  \label{eq:constant:on:equivIN}
  f(y)=f(y')\qquad\text{ whenever } y\equivIN y'\;.
\end{equation}
For algorithms defined by random transformations, the law of large numbers strengthens to:\nolinebreak
\begin{corollary}
  \label{result:estimate:f:of:y}
  Suppose a sampling algorithm $S$ satisfies \eqref{conditions:prefix:sampler}, and 
  ${f:\OUT\rightarrow\mathbb{R}}$ is a Borel function satisfying
  \eqref{eq:constant:on:equivIN}. Require $\AlgInf{\infty}(y)$ is $T(\group)$-ergodic.
  Let $(f_m)$ be a sequence of functions on $\OUT$. 
  Then for every $y$ with (i) ${f\in\Lone(P_y)}$ and (ii) ${f_m\rightarrow f}$ $P_y$-a.s.,
  \begin{equation*}
    \frac{1}{|\group_k|}\sum_{\phi\in\group_k}f_k(\AlgInf{\infty}(y))\xrightarrow{k\rightarrow\infty}
    f(y)\qquad P_y\text{-a.s.}
  \end{equation*}
  If $y$ is replaced by a $\subIN$-valued random variable $Y$, and (i) and (ii) hold $\Law(Y)$-a.s.,
  convergence holds $\Law(Y)$-a.s.
\end{corollary}

\subsection{The topology induced by a sampling algorithm}

Any sampling algorithm $S$ whose prefix densities exist on a set ${\subIN}$ induces a topology on this set,
the weak topology of $\mathbf{\hom}$ (i.e.\ the smallest topology on $\subIN$ that makes each
prefix density ${t_{x_k}:\subIN\rightarrow[0,1]}$ continuous). Informally speaking, if 
the equivalence classes of $\equivIN$ coincide with the fibers of $\mathbf{\hom}$ (as is the case
in \cref{result:direct:union:sampler}),
this is the smallest topology that distinguishes input points whenever they
are distinguishable by the sampler.
If $S$ is defined by \cref{alg:uvertex:simple}, 
the prefix
densities are precisely the ``homomorphism densities'' of graph limit theory---depending on the definition,
possibly up to normalization \citep{Diaconis:Janson:2007}. The weak topology of ${\mathbf{\hom}}$ is hence
the cut norm topology \citep{Borgs:Chayes:Lovasz:Sos:Vesztergombi:2008}. 
The cut norm topology is defined on the set of undirected,
simple graphs with vertex set $\mathbb{N}$, and coarsens the almost discrete topology on this set.
One may hence ask how this property depends on the
sampler: \emph{Under what conditions on the subsampling algorithm does the topology induced by the sampler
coarsen the topology of the input space?} If the algorithm is defined by random transformation
as above, that is always the case:
\begin{proposition}
  \label{result:continuous:prefix:densities}
  Let $S$ be a sampling algorithm defined as in \eqref{eq:random:transformation:sampling} by a prefix action $T$
  on an almost discrete space $\OUT$. Let ${\subIN}$ be any topological subspace
  of $\OUT$ such that the prefix densities exist
  for each ${y\in\subIN}$. Then $\mathbf{\hom}$ is continuous on $\subIN$.
\end{proposition}

\section{Selecting vertices independently}
\label{sec:sampling:vertices}

Throughout this section, we choose both input space $\IN$ and the output space $\OUT$ as
the set of simple, undirected graphs with vertex set $\bbN$, and
$\textrestk{\argdot}$ extracts the induced subgraph on the first $k$ vertices.

\subsection{Exchangeability and graphons}
\label{sec:graphons}

\cref{alg:uvertex:simple} selects a subgraph uniformly from the set of all subgraphs
of size $n$ of the input graph $\textrestn{y}$. Such uniform random subgraphs are integral
to the definition of graphons 
\citep{Borgs:Chayes:Lovasz:Sos:Vesztergombi:2008,Borgs:Chayes:Lovasz:Sos:Vesztergombi:2012:1}, and the prefix densities are in this case precisely the \emph{homomorphism densities} of graph limit theory (up to
normalization). It is thus a well-known fact that \cref{alg:uvertex:simple} induces the
class of graphon models, whose relationship to exchangeable random graphs has in turn be clarified
by \citet{Diaconis:Janson:2007} and \citet{Austin:2008}.

Applied to this case, our results take the following form: 
\cref{alg:uvertex:simple} can equivalently be represented as a random transformation
\eqref{eq:random:transformation:sampling}. Define $T$ as the action of 
$\FSym$ that permutes the vertex labels of a graph, and rewrite \cref{alg:uvertex:simple} as:
\addtocounter{algorithm}{-1}
\begin{algorithmdash}
  \label{alg:uvertex:random:trafo}
  \begin{tabular}{rl}
    \text{{\small i.})} & \text{Draw ${\Phi_n\sim\text{Uniform}(\Sym{n})}$.}\\
    \text{{\small ii.})} & \text{Generate the permuted graph ${X_n:=\Phi_n(\textrestn{y})}$.}\\
    \text{{\small iii.})} & \text{Report the subgraph ${\Alg{n}{k}(y):=\textrestk{X_n}}$.}\\
  \end{tabular}
\end{algorithmdash}
Clearly, \cref{alg:uvertex:random:trafo} and \ref{alg:uvertex:simple} are equivalent.
It is possible
to construct pathological input graphs $y$ for which prefix
densities do not exist; we omit details, and simply define
${\subIN:=\braces{y\in\IN\,|\,\text{\cref{alg:uvertex:random:trafo} has prefix densities}}}$.
Then $\subIN$ is invariant under $\FSym$, and we obtain from Theorems \ref{lemma:main} and
\ref{result:direct:union:sampler}:
\begin{corollary}
  \cref{alg:uvertex:simple} is idempotent, and the limiting random graph 
  ${\AlgInf{\infty}(y)}$ is exchangeable.
  Let ${f\in\Lone(P_y)}$ be a function constant on each equivalence class of $\equivIN$. 
  Then if functions ${f_k:\OUT_k\rightarrow\mathbb{R}}$ satisfy
  ${f_k(\textrestk{x})\rightarrow f(x)}$ for all $x$ outside a $P_y$-null set, 
  \begin{equation*}
    \frac{1}{k!}\sum_{\pi\in\Sym{k}}f_k(T_{\pi}(\AlgInf{k}(y))
    \quad\xrightarrow{k\rightarrow\infty}\quad
    f(y)\qquad\text{ almost surely.}
  \end{equation*}
  The equivalence classes of $\equivIN$ are the fibers of $\mathbf{\hom}$.
\end{corollary}

Let $w$ be a graphon, 
\ie a measurable function ${w:[0,1]^2\rightarrow[0,1]}$ symmetric in its
arguments \citep{Borgs:Chayes:Lovasz:Sos:Vesztergombi:2008}. 
Let $X_w$ be a random graph with the canonical distribution defined by $w$: The
(symmetric) adjacency matrix of $X_w$ is given by
\begin{equation*}
  \bigl(\mathbb{I}\braces{U_{ij}<w(U_i,U_j)}\bigr)_{i<j\in\bbN}
  \qquad\text{ where }(U_i)_{i\in\bbN}\text{ and }(U_{ij})_{i,j\in\bbN}\text{ are i.i.d. }\text{Uniform}[0,1]\;.
\end{equation*}
Let $t_w$ denote the (suitably normalized) vector of homomorphism densities of $w$ 
\citep{Borgs:Chayes:Lovasz:Sos:Vesztergombi:2008,Diaconis:Janson:2007}. 
Comparing the definitions of homomorphism and prefix densities, we have 
${\mathbf{\hom}(X_w)=t_w}$ almost surely. Since the fibers of $\mathbf{\hom}$ are the
equivalence classes of $\equivIN$, we can choose a fixed graph ${y\in\mathbf{\hom}^{-1}(t_w)}$,
and obtain
\begin{equation*} 
  X_w\equivIN y \quad\text{ a.s. }
  \qquad\text{ and }\qquad
  X_w\equdist\AlgInf{\infty}(X_w)\equdist\AlgInf{\infty}(y)\;.
\end{equation*}
For every graphon $w$, there is hence a graph $y$ such that ${X_w\equdist\AlgInf{\infty}(y)}$.
It is well-known that the law of  $X_w$ remains unchanged if a Lebesgue-measure preserving
transformation $\psi$ of ${[0,1]}$ is applied to $w$: If ${w'=w\circ(\psi\otimes\psi)}$, 
then ${X_{w'}\equdist X_w}$. Equivalence classes of graphons hence correspond to equivalence
classes of graphs under $\equivIN$,
\begin{equation*}
  X_{w'}
  \quad\equivIN\quad
  X_w
  \quad\equivIN\quad
  \AlgInf{\infty}(X_w)\qquad\text{ almost surely.}
\end{equation*}

\subsection{Misspecification of graphon models as a sample selection bias}
\label{sec:misspecification}

\begin{figure}
  \makebox[\textwidth][c]{
    \begin{tikzpicture}
      \node[rotate=270] at (0,0) {\includegraphics[width=4cm]{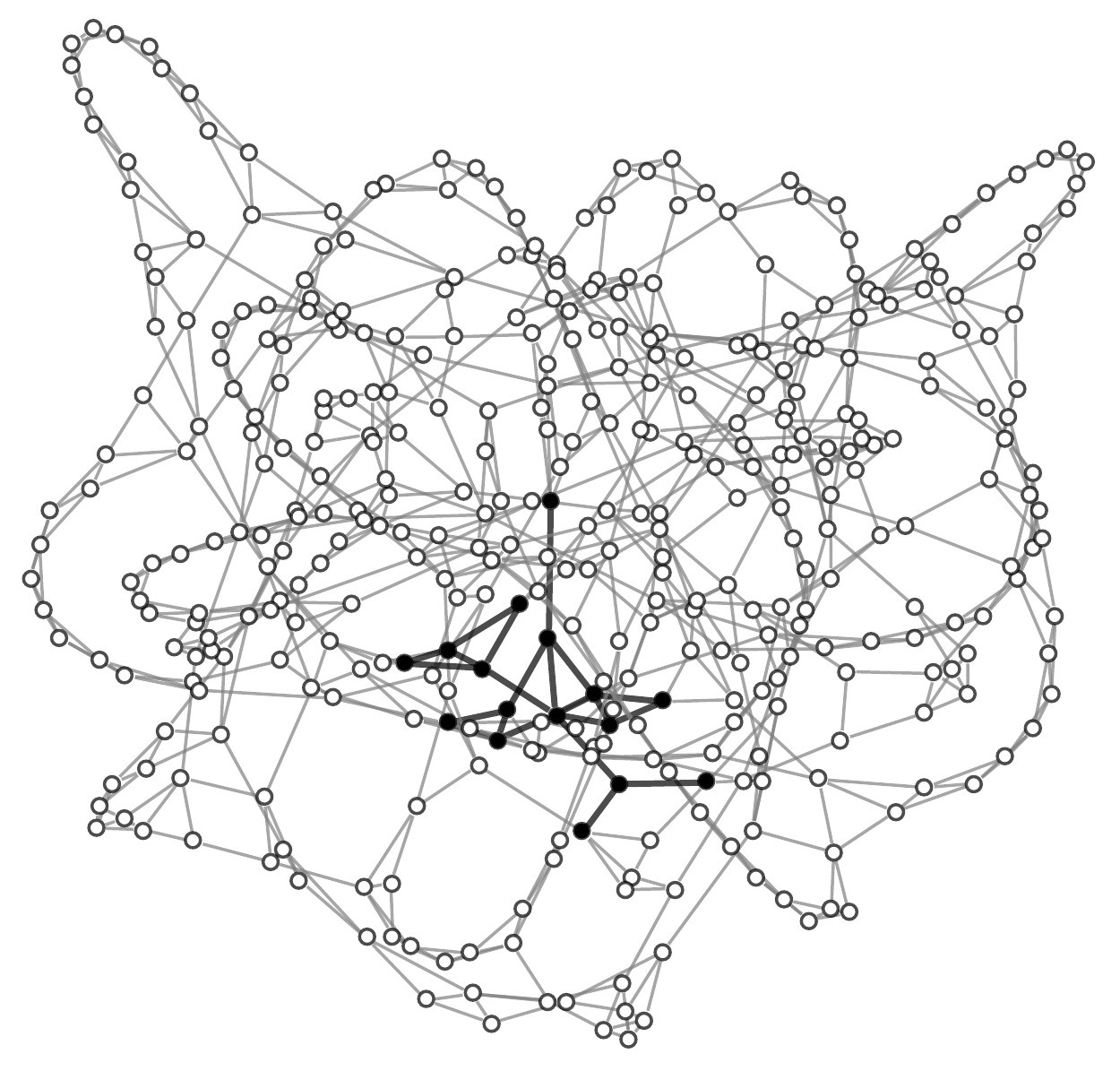}};
      \node[rotate=270] at (4,0) {\includegraphics[width=5cm]{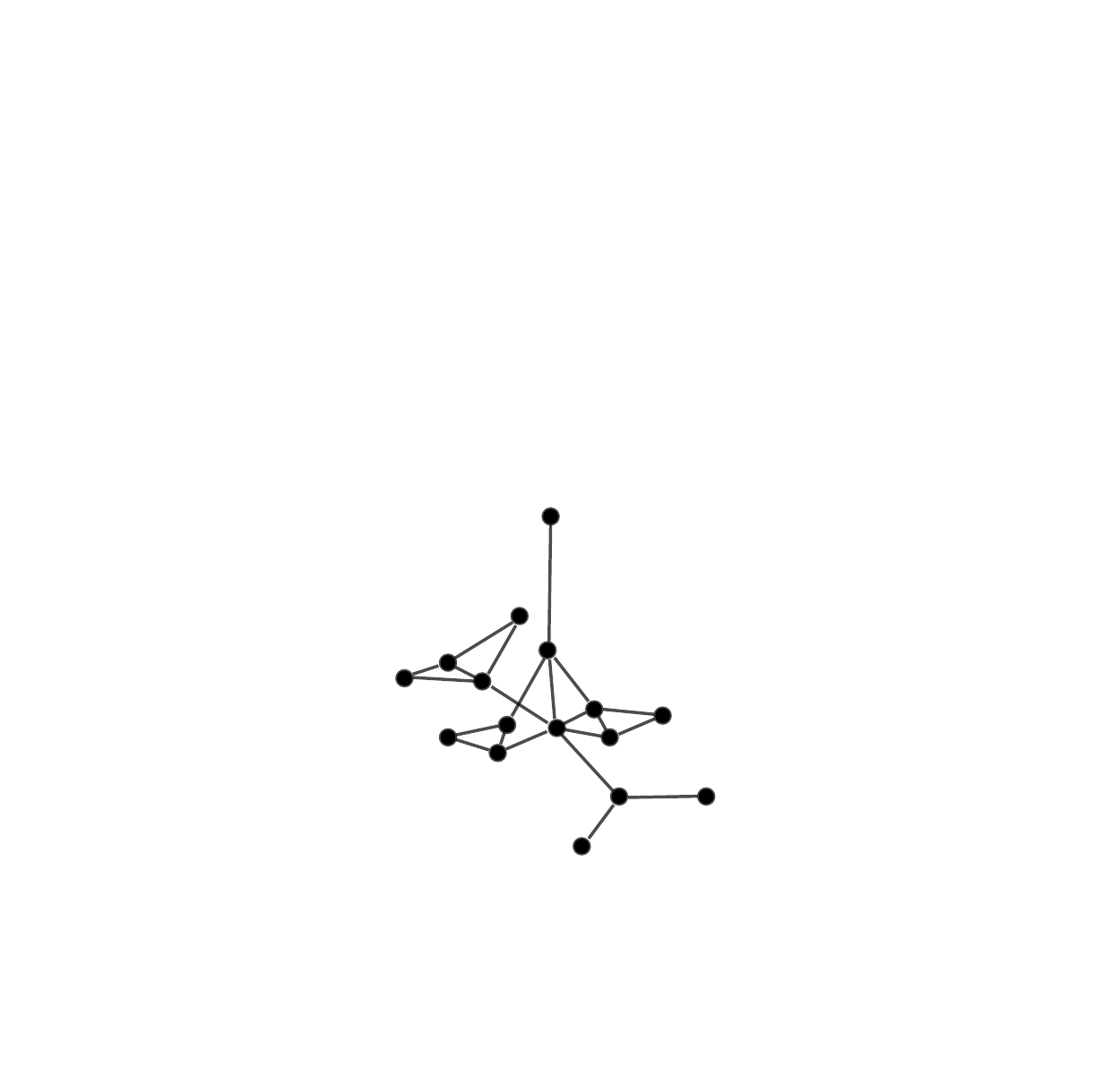}};
      \node at (6.2,0) {\includegraphics[width=3cm]{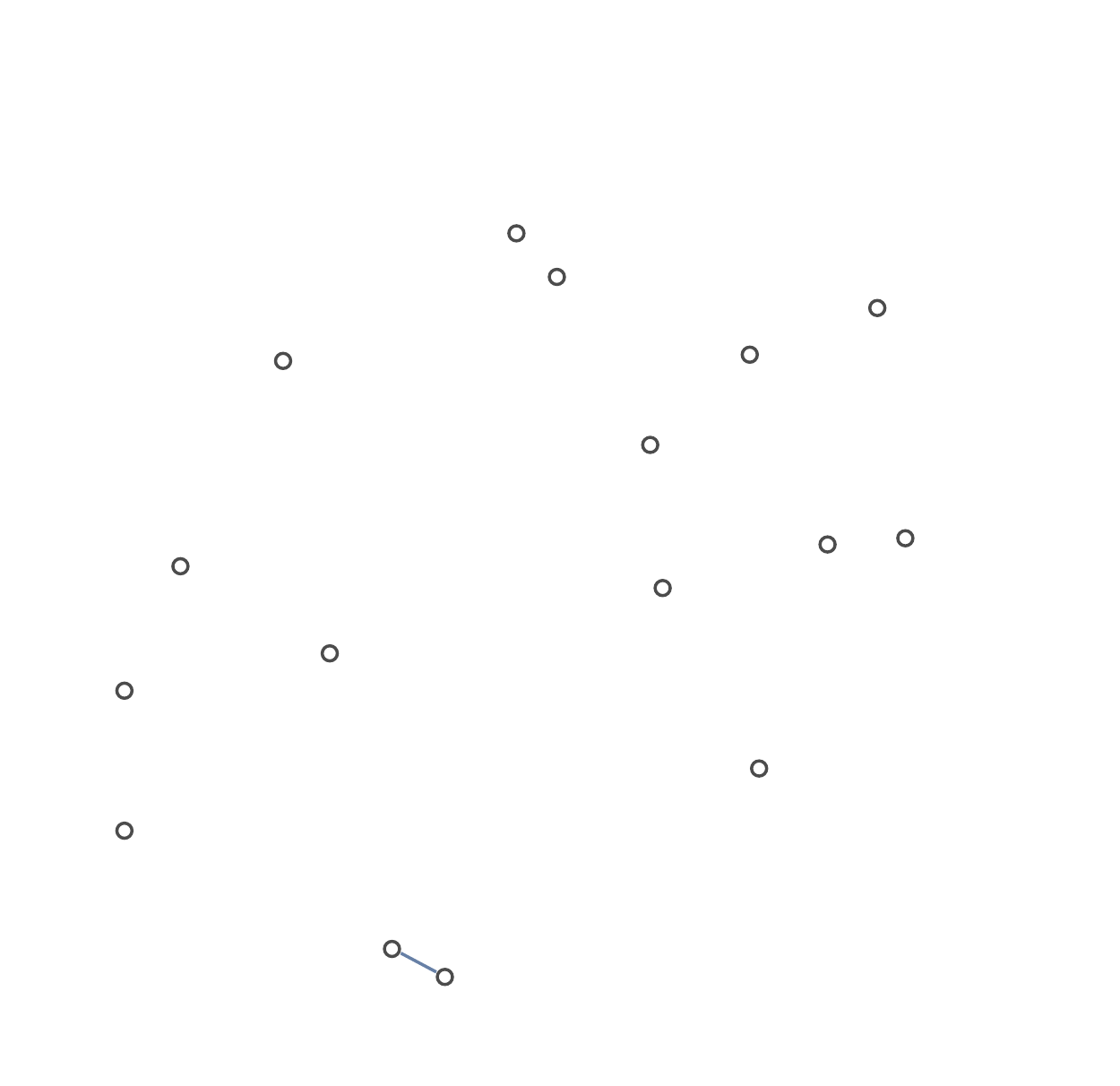}};
      \node at (10,0) {\includegraphics[width=.3\textwidth]{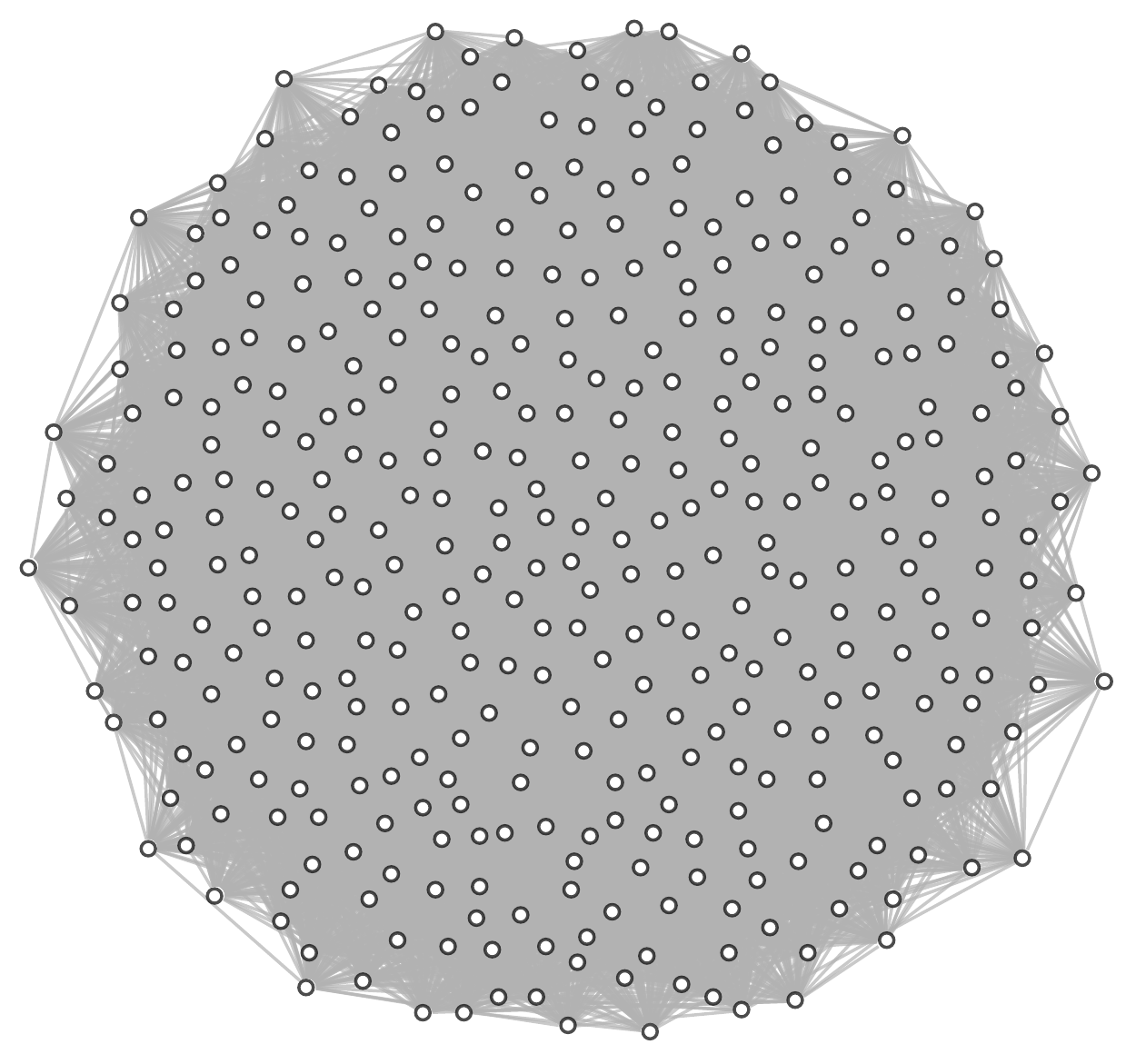}};

      \node at (0,-2.5) {\scriptsize (i)};
      \node at (3.2,-2.5) {\scriptsize (ii)};
      \node at (6.2,-2.5) {\scriptsize (iii)};
      \node at (10,-2.5) {\scriptsize (iv)};
    \end{tikzpicture}
  }
  \caption{\emph{(i)} Input graph. \emph{(ii)} A subsample $X_k$ generated by an algorithm that extracts
    the $n$-neighborhood of a random vertex, here ${n=2}$. The sample $X_k$ has 16 vertices.
  \emph{(iii)} A sample of the same size generated by \cref{alg:uvertex:simple}. 
  \emph{(iv)} Reconstruction of the input graph (i), \ie a sample of with the same number of vertices
  as (i), from a graphon model fit to the sample (ii).}
  \label{fig:misspecification:1}
\end{figure}
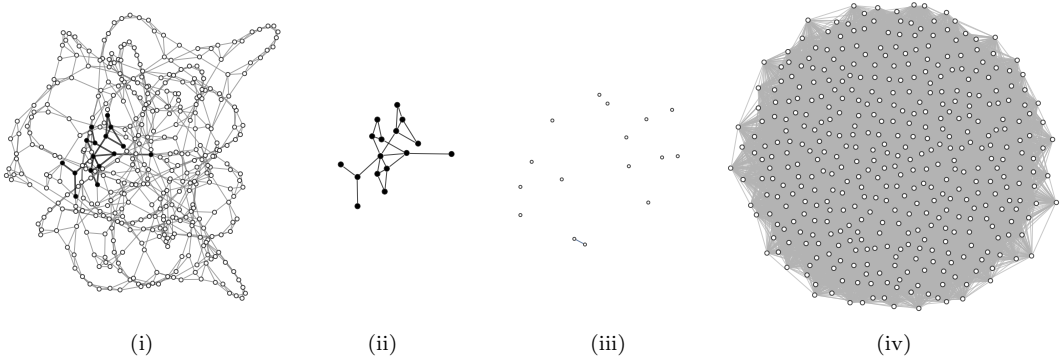

A model $\mathcal{P}$ is \kword{misspecified} for data generated by a random variable $X$ if
${\Law(X)\not\in\mathcal{P}}$. If $X$ is with positive probability a sparse graph, that implies,
without any further knowledge of $X$, that a graphon model is misspecified for $X$. This fact
is well known \citep[e.g.][]{Janson:2016:1,Caron:Fox:2014:1,Orbanz:Roy:2014}.

A \kword{sample selection bias} is an erroneous assumption on what properties of the underlying
population a sample is representative of. 
One way in which such biases occur is if the sampling protocol is known, but either its behavior
or the underlying population are insufficiently understood---for example, in opinion polling
problems if the protocol has a systematic tendency to exclude individuals with certain properties.
Another is if false assumptions are made about which sampling
protocol has been used. 
A sampling algorithm $\AlgInf{\infty}$ with input set $\subIN$
induces a model ${\model=\braces{P_y|y\in\subIN}}$. 
Analyzing data generated by another algorithm
$\AlgInf{\infty}'$ using $\model$ thus constitutes a selection bias. 
Such a selection bias results in misspecification if 
${\AlgInf{\infty}'(y)\not\in\mathcal{P}}$.

Explicitly considering sampling makes the nature of graphon misspecification more precise:
As noted above, \cref{alg:uvertex:simple} induces the model ${\mathcal{P}=\braces{\Law(X_w)|w\text{ graphon}}}$.
Explaining data by a graphon model hence implicitly assumes the data is an outcome
of \cref{alg:uvertex:simple}. One might
argue, for example, that the distinction between sparse and dense is of limited relevance for small graphs,
and that a small sample can hence be fit without concern using a graphon model. 
\cref{fig:misspecification:1} illustrates the implicit sampling assumption assumption can have
drastic consequences, even for small samples.

Consider a sample ${X_k:=\AlgInf{k}(y)}$ generated by the limit of \cref{alg:uvertex:simple}
from an infinite input graph $y$.
If $X_k$ contains a subgraph $x_j$ of size ${j\leq k}$,
$y$ must contain infinitely
many copies of $x_j$. There are $\binom{k}{j}$ possible subgraphs of size $k$ in $X_k$.
If $x_j$ occurs $m$ times in $X_k$, a graphon model assumes a proportion
${m/\binom{k}{j}}$ of all subgraphs of size $j$ in $y$ match $x_j$.
A consequence of the rapid growth of $\binom{k}{j}$ is:
\begin{itemize}
\item Small observed patterns are assigned much higher probability than larger ones.
\end{itemize}
For example, if a sample $X_{20}$ contains a subgraph $x_k$ exactly once, a graphon
reconstruction of $X_{20}$ contains $x_k$ with probability ${1/1140}$ if ${j=3}$,
compared to ${1/38760}$ if ${j=6}$.

A single edge in $X_k$ is a subgraph of size two. An \emph{isolated} edge, however, is
a subgraph of size $k$: One edge is present, all other edges between its terminal vertices
and the remaining graph are absent. A further consequence of the above is hence:
\begin{itemize}
\item Graphon models tend not to reproduce (partially) isolated subgraphs.
\end{itemize}
That is illustrated by \cref{fig:misspecification:2}, which compares a protein-protein 
interaction network to its reconstruction from a graphon model. The semi-isolated chains
in the input graph, for example, are not visible in the reconstruction; they are present, 
but not isolated.

\begin{figure}
\begin{center}
  \begin{tikzpicture}
    \path[use as bounding box] (-1,-1.7) rectangle (7,2);
    \node at (0,0) {\includegraphics[width=4cm,angle=-20]{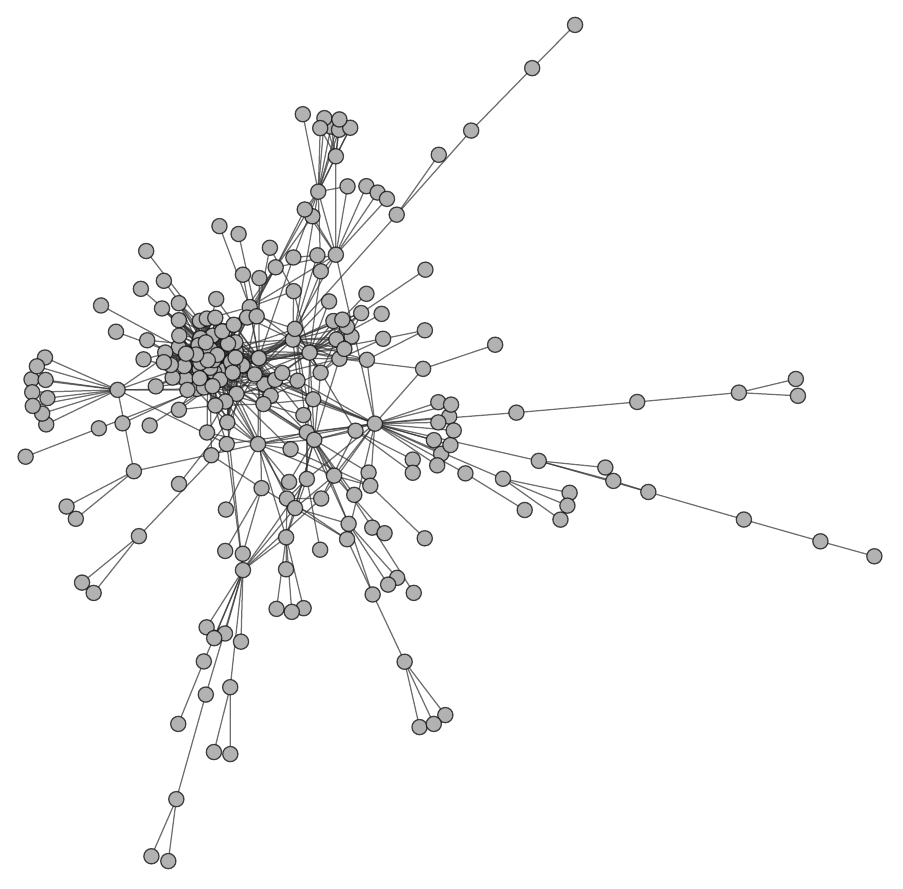}};
    \node[rotate=180] at (6,0) {\includegraphics[width=4cm,angle=90]{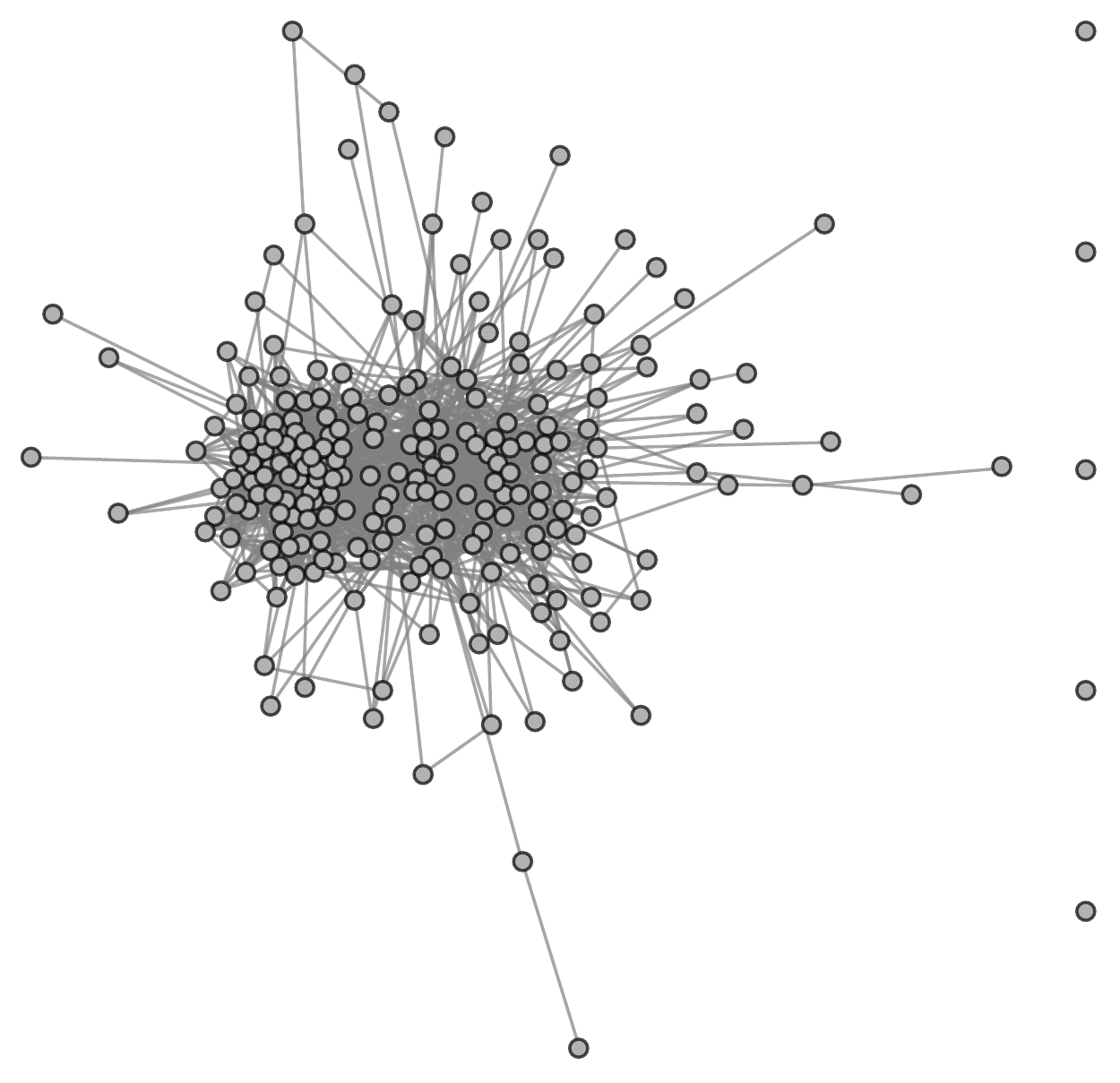}};
  \end{tikzpicture}
\end{center}
\caption{A protein-protein interaction network (left) and its reconstruction
  (a sample with the same number of vertices) from a graphon model (right).
}
\label{fig:misspecification:2}
\end{figure}

\subsection{Sparsified graphon models}

To address denseness, ``sparsified'' graphon models have
been proposed as a remedy, and recent work in mathematical statistics frequently 
invokes these random graphs as network models.
They are equivalent to random graphs originally introduced in \citep{Bollobas:Janson:Riordan:2007},
and are defined as follows:
Fix a graphon $w$ and a monotonically decreasing function ${\rho:\bbN\rightarrow[0,1]}$.
A graph of size $k$ is then generated as
\begin{equation*}
  \bigl(\mathbb{I}\braces{U_{ij}<\rho(k)w(U_i,U_j)}\bigr)_{i<j\leq k}
  \qquad\text{ where }(U_i)_{i\leq k}\text{ and }(U_{ij})_{i,j\leq k}\text{ are i.i.d. }\text{Uniform}[0,1]\;.
\end{equation*}
A graph of this form can equivalently be generated by generating $k$ vertices of the
graph $X_w$ defined by the graphon, followed \iid bond percolation, where each present
edge is deleted independently with probability $1-\rho(k)$. This model is hence generated
by the following sampling algorithm: 
\begin{algorithm}
  \label{alg:uvertex:BJR}
  \begin{tabular}{rl}
    \text{{\small i.})} & \text{Select $k$ vertices of $\textrestk{y}$ 
      independently and uniformly without replacement.}\\
    \text{{\small ii.})} & \text{Extract the induced subgraph ${\Alg{n}{k}(\textrestn{y})}$ of 
      $\textrestn{y}$ on these vertices.}\\
    \text{{\small iii.})} & \text{Label the vertices of ${\Alg{n}{k}(\textrestn{y})}$ 
      by ${1,\ldots,k}$ in order of appearance.}\\
    \text{{\small iv.})} & \text{Delete each edge in ${\Alg{n}{k}(\textrestn{y})}$ 
      independently with probability $p(k)$.}\\
  \end{tabular}
\end{algorithm}
Although the random graphs generated by such models can be made sparse by suitable
choice of $\rho$, this modification clearly does not alleviate the misspecification
problems discussed above. To emphasize:
\begin{center}
  \emph{Sparsified graphon models still assume
    implicitly that the input graph is dense, and explain sparsity by assuming the data
    has been culled after it was sampled. }
\end{center}
Note the application of \cref{alg:uvertex:simple} and the bond percolation step (iv) in 
\cref{alg:uvertex:BJR} cannot be exchanged: One cannot equivalently assume the input graph
has been sparsified first, and then sampled.

\subsection{Graphon models on $\sigma$-finite spaces}
\label{sec:KEGs}

Generating a draw from a graphon model involves a generic sequence of \iid random variables.
The joint law of these variables is a probability measure.
\citet{Caron:Fox:2014:1} have proposed a random graph model that, loosely speaking, 
samples from a $\sigma$-finite measure by substituting a Poisson process, which makes it possible
to generate sparse graphs. \citet{Veitch:Roy:2015:1}, \citet*{Borgs:Chayes:Cohn:Holden:2016:1}
and \citet{Janson:2016:1,Janson:2017:1} have extended this idea to
a generalization of graph limits that can represent certain sparse graphs.
It can be shown that this approach substitutes \cref{alg:uvertex:simple}
by a form of site percolation: 
\begin{algorithm}
  \label{alg:VR}
  \begin{tabular}{rl}
    \text{{\small i.})} & \text{Select each vertex in $\textrestn{y}$ independently,
      with a fixed probability ${p\in[0,1]}$.}\\
    \text{{\small ii.})} & \text{Extract the induced subgraph ${X_{m}}$ of 
      $\textrestn{y}$ on the selected vertices.}\\
    \text{{\small iii.})} & \text{Delete all isolated vertices
      from $X_{m}$, and report the resulting graph.}\\
  \end{tabular}
\end{algorithm}
Note the size of the output graph is now random.
\cref{alg:VR} is derived by \citet{Veitch:Roy:2016:1} (who refer to the site percolation
step as \emph{$p$-sampling})
as the sampling scheme that describes the relation between $\sigma$-finite graphon models at 
different sizes. \citet*{Borgs:Chayes:Cohn:Veitch:2017:1}
show it defines the model class, and characterize the associated topology.

\subsection{Biasing by degree}
\label{sec:deg:bias}

\cref{alg:uvertex:simple} and \ref{alg:uvertex:BJR} fail to resolve sparse input graphs
since they select vertices independently of the edge set.
\cref{alg:VR} circumvents the problem by oversampling---it selects a much
larger number of vertices, and then weeds out insufficiently salient bits.
We compare this to an example that still selects vertices independently, but includes
some information about the edge set, in the form of the vertex degrees:
\begin{algorithm}
  \label{alg:dbiased}
  \begin{tabular}{rl}
    \text{{\small i.})} & \text{Select $n$ vertices of $\textrestn{y}$ 
      independently w/o replacement from the}\\ 
    & \text{degree-biased distribution.}\\
    \text{{\small ii.})} & \text{Extract the induced subgraph ${S_n(\textrestn{y})}$ of 
      $\textrestn{y}$ on these vertices.}\\
    \text{{\small iii.})} & \text{Label the vertices of ${S_n(\textrestn{y})}$ 
      by ${1,\ldots,n}$ in order of appearance.}\\
  \end{tabular}
\end{algorithm}
We observe immediately this algorithm is not idempotent:
Suppose an input graph $y_4$ for ${\Alg{4}{3}}$ is chosen as follows:
\begin{center}
  \begin{tikzpicture}[scale=.8]
    \begin{scope}[xshift=-3.5cm,rotate=0]
      \node[circle,draw,scale=.7,fill=gray!30!white] (v1) at (0,0) {\scriptsize 1};
      \node[circle,draw,scale=.7,fill=gray!30!white] (v2) at (0,-1) {\scriptsize 2};
      \node[circle,draw,scale=.7,fill=gray!30!white] (v3) at (-1,-1) {\scriptsize 3};
      \node[circle,draw,scale=.7,fill=gray!30!white] (v4) at (1,-1) {\scriptsize 4};
      \draw (v1)--(v2); \draw (v2)--(v3); \draw (v2)--(v4);
      \node at (0,-1.65) {$y_4$};
    \end{scope}
    \begin{scope}
      \node[circle,draw,scale=.7,fill=gray!30!white] (v1) at (0,0) {\scriptsize 2};
      \node[circle,draw,scale=.7,fill=gray!30!white] (v2) at (1.44,0) {\scriptsize 1};
      \node[circle,draw,scale=.7,fill=gray!30!white] (v3) at (.72,-1) {\scriptsize 3};
      \draw (v1)--(v2); \draw (v2)--(v3);
      \node at (.72,-1.65) {$x_3$};
    \end{scope}
    \begin{scope}[xshift=2.5cm]
      \node[circle,draw,scale=.7,fill=gray!30!white] (v1) at (0,0) {\scriptsize 2};
      \node[circle,draw,scale=.7,fill=gray!30!white] (v2) at (1.44,0) {\scriptsize 1};
      \node[circle,draw,scale=.7,fill=gray!30!white] (v3) at (.72,-1) {\scriptsize 3};
      \draw (v1)--(v3); \draw (v2)--(v3);
      \node at (.72,-1.65) {$x'_3$};
    \end{scope}
  \end{tikzpicture}
\end{center}
If ${\Alg{4}{3}}$ is defined by \cref{alg:uvertex:simple}, it generates
${x_3}$ and ${x_4}$ with equal probability. Under \cref{alg:dbiased}, the probabilities differ.
See also \cref{fig:dbiased:interactome}. This effect becomes more pronounced in the
input size limit, and permits the limiting algorithm $\AlgInf{\infty}(y)$ to distinguish sparse
from empty input graphs. Consider an input graph $y$ with vertex set $\mathbb{N}$,
defined as follows:
\begin{equation}
  \label{eq:graph:dbiased:resolution}
  \begin{split}
\begin{tikzpicture}
    \begin{scope}[scale=1]
    \node at (-.5,0) {};
    \node[circle,scale=.65,thick,draw,fill=gray!30!white,label=left:{\scriptsize $1$}] (u1) at (0,1) {}; 
    \node[circle,scale=.65,thick,draw,fill=gray!30!white] (v2) at (1,0) {};    
    \node[circle,scale=.65,thick,draw,fill=gray!30!white] (v3) at (2,0) {};    
    \node[circle,scale=.65,thick,draw,fill=gray!30!white] (v4) at (3,0) {};    
    \node[circle,scale=.65,thick,draw,fill=gray!30!white] (v5) at (4,0) {};    
    \node[circle,scale=.5,thick] (v6) at (5,0) {$\cdots$};    
    \node at ($(v2)+(0,-.5)$) {\scriptsize $2$};
    \node at ($(v3)+(0,-0.5)$) {\scriptsize $3$};
    \node at ($(v4)+(0,-0.5)$) {\scriptsize $4$};
    \node at ($(v5)+(0,-0.5)$) {\scriptsize $5$};
    \draw (u1)--(v2);
    \draw (u1)--(v3);
    \draw (u1)--(v4);
    \draw (u1)--(v5);
    \draw[dotted] (u1)--(v6.north west);
    \end{scope}
\end{tikzpicture}
  \end{split}
\end{equation}
In $\textrestn{y}$, vertex $1$ has degree ${m-1}$; all other vertices have degree 1.
Under \cref{alg:uvertex:simple}, the probability
of vertex 1 being selected converges to 0 as ${n\rightarrow\infty}$. Consequently, 
${\AlgInf{\infty}(y)}$ is empty almost 
surely. If \cref{alg:dbiased} is used instead, each step ${\AlgInf{m}(y)}$ selects
vertex $1$ with probability ${\geq 1/2}$ until it is selected, and 
${\AlgInf{\infty}(y)}$ is almost surely connected.
Thus, \cref{alg:dbiased} resolves sparse graphs.

On the other hand, analysis of the algorithm becomes considerably more complicated
than for \cref{alg:uvertex:simple}. It is not clear, for example, which input graphs have
prefix density, although the example of the graph in \eqref{eq:graph:dbiased:resolution} shows
graphs with prefix densities exist.

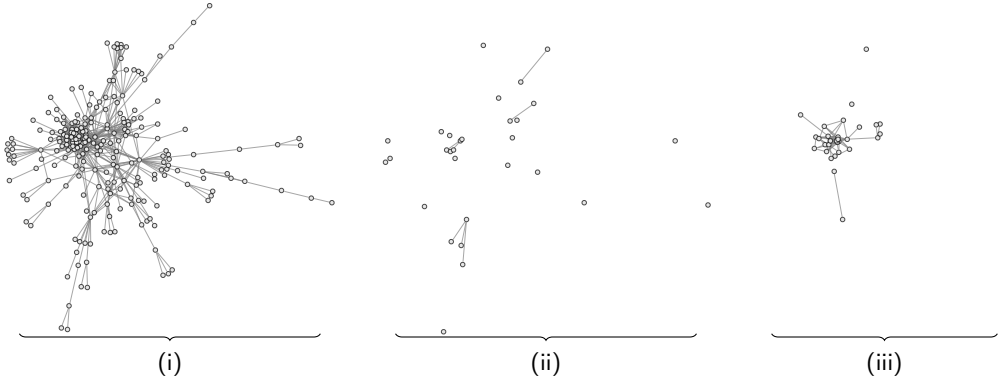
\begin{figure}
  \begin{center}
        \begin{tikzpicture}[mybraces,font=\sffamily]
      \node at (-5,0) {\includegraphics[width=.3\textwidth]{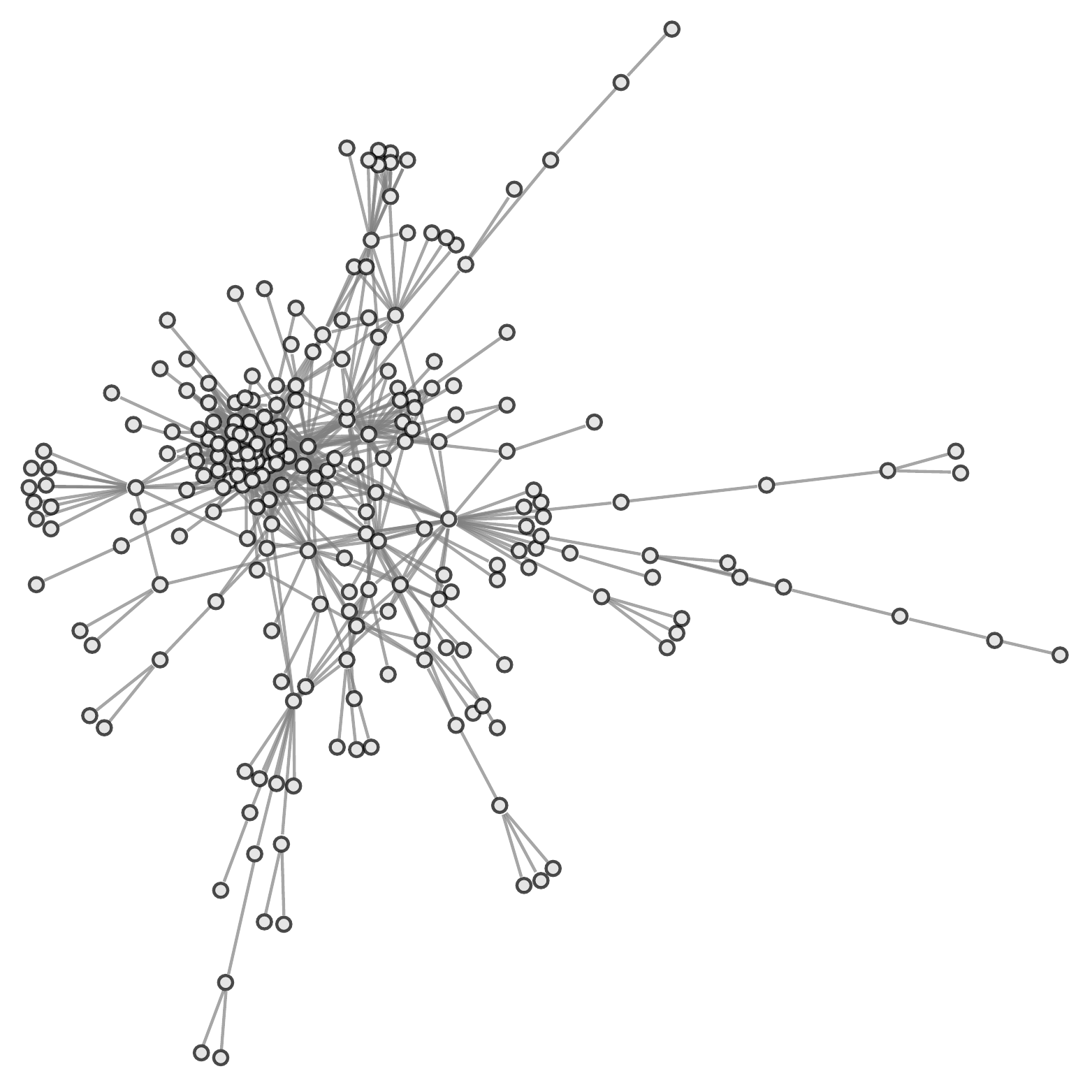}};
      \node at (0,0) {\includegraphics[width=.3\textwidth]{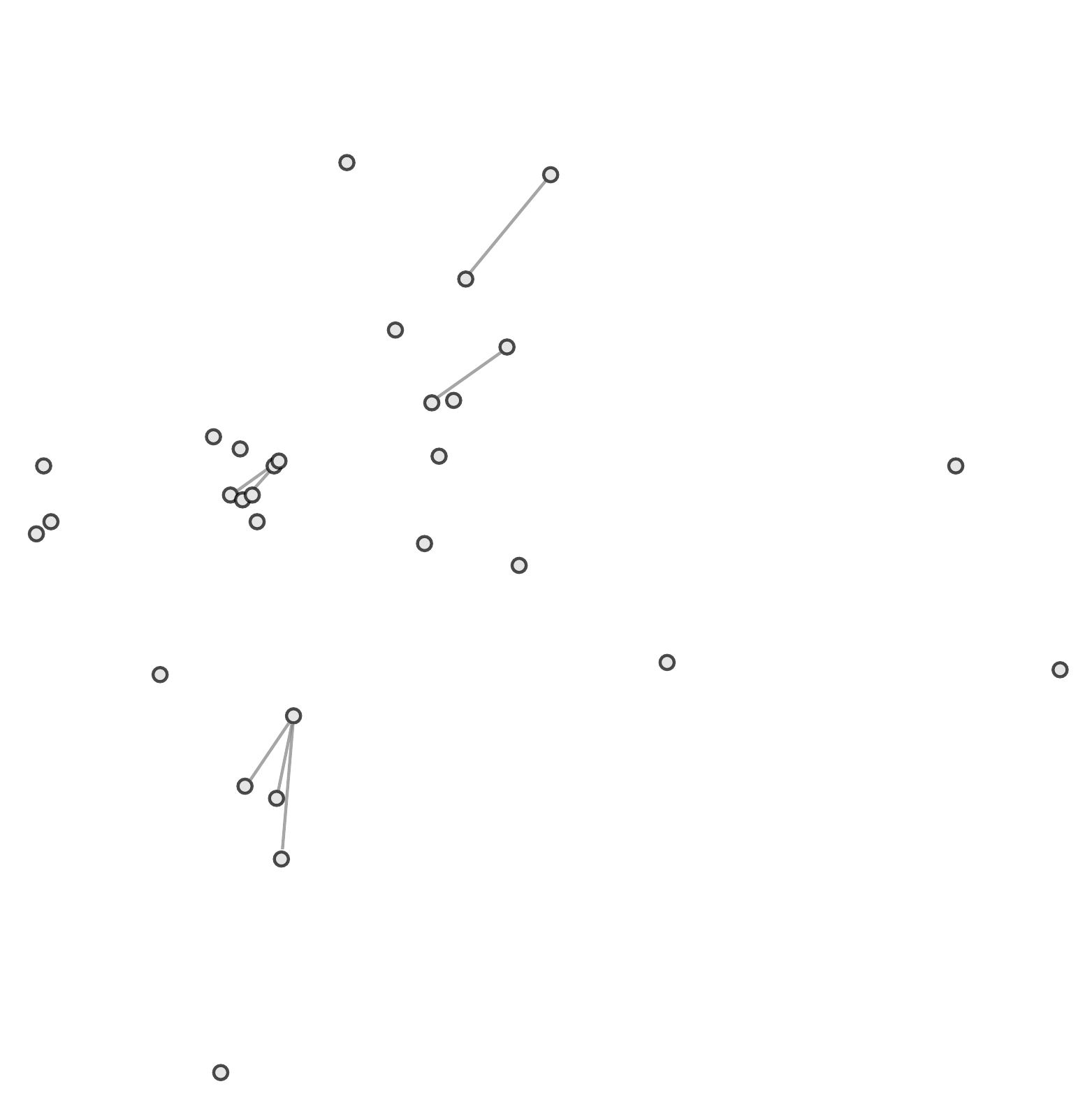}};
      \node at (5,0) {\includegraphics[width=.3\textwidth]{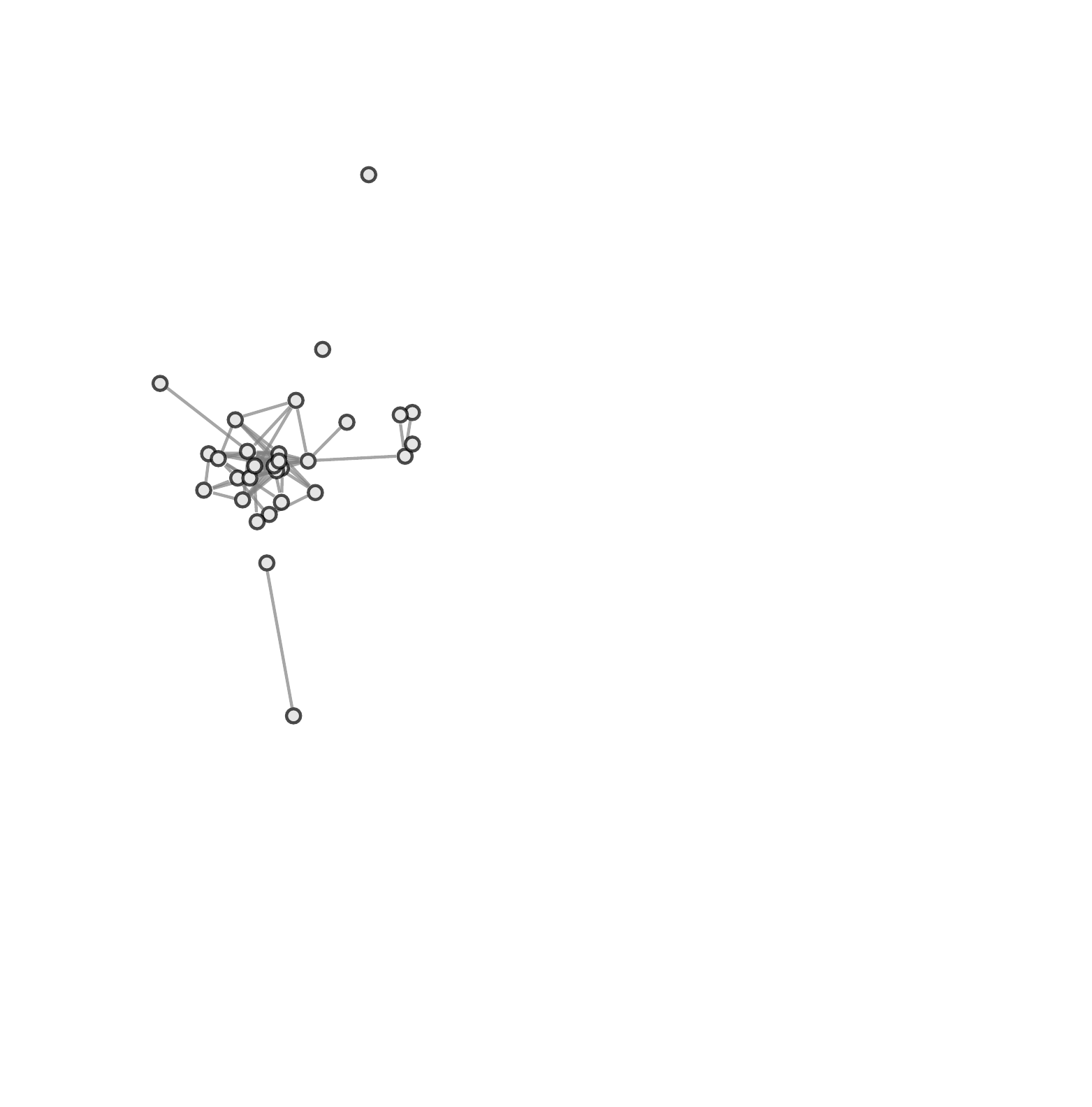}};
      \draw[mirrorbrace] (-7,-2.2)--(-3,-2.2);
      \draw[mirrorbrace] (-2,-2.2)--(2,-2.2);
      \draw[mirrorbrace] (3,-2.2)--(6,-2.2);
      \node at (-5,-2.6) {(i)};
      \node at (0,-2.6) {(ii)};
      \node at (4.5,-2.6) {(iii)};
    \end{tikzpicture}
  \end{center}
  \caption{Biasing by degree. (i) Input graph. (ii) A uniform subsample generated by 
    \cref{alg:uvertex:simple}. (iii) A degree-biased subsample of the same size, generated by \cref{alg:dbiased}.
  }
  \label{fig:dbiased:interactome}
\end{figure}

\subsection{Reporting a shortest path}
\label{sec:shortest:path}

In a sparse input graph, uniformly selected vertices are not connected by an edge
with (limiting) probability 1, but they may be connected by a path of finite length.
A possible way sample vertices independently of the edge \emph{and} to resolve sparse graphs 
is hence to report path length, rather than presence or absence of edges.
Choose ${\subIN\subset\IN}$ as the set of undirected, simple graphs that are connected and have finite
diameter. In other words, these are graphs on $\bbN$ in which any two vertices are connected by a path of
finite length.
\begin{algorithm}
  \label{alg:shortst:path}
  \begin{tabular}{rl}
    \text{{\small i.})} & \text{Select $n$ vertices of $\textrestn{y}$ 
      independently and uniformly w/o replacement.}\\
    \text{{\small ii.})} & \text{Choose ${\Alg{n}{k}(y)}$ as the complete
      graph on $n$ vertices, and mark each edge}\\
    & \text{$(i,j)$ by the length of the shortest
      path between $i$ and $j$ in $\textrestn{y}$.}\\
    \text{{\small iii.})} & \text{Label the vertices of ${\Alg{n}{k}(y)}$ 
      by ${1,\ldots,n}$ in order of appearance.}\\
  \end{tabular}
\end{algorithm}
The example graph in  \eqref{eq:graph:dbiased:resolution} again illustrates this algorithm
can resolve sparse graphs.

In this case, ${\IN\neq\OUT}$, since the output graphs have weighted edges. The algorithm does generate
exchangeable output: If ${\Alg{n}{k}(y)}$ is regarded as a symmetric, ${n\times n}$ adjacency matrix
with values in $\bbN$, its distribution is invariant under joint permutations of rows and columns.
Thus, the output is a jointly exchangeable array. Since ${\IN\neq\OUT}$, the algorithm cannot be represented
as a random permutation, and \cref{result:direct:union:sampler} is not directly applicable. 
One can, however, define a map ${h:\subIN\rightarrow\OUT}$ that takes 
an infinite graph $y$ to a graph $h(y)$ on the same vertex set, with each edge marked by the corresponding
shortest path. Then \cref{alg:shortst:path} is equivalent to application of 
\cref{alg:uvertex:simple} to ${h(y)}$ (where \cref{alg:uvertex:simple} additionally reports edge marks).

\section{Subsampling sequences and partitions}
\label{sec:partitions}

\def\Part{\Pi}
\def\part{\pi}
\def\simpart{\stackrel{\tiny\pi}{\sim}}

If a sampler selects edges rather than vertices of a graph, it produces a sequence of edges.
If this sequence is exchangeable, the output graph properties are closely related to those of 
exchangeable sequences and partitions. It is hence helpful to first consider sequences and partitions from a 
subsampling perspective, even though the results so obtained are known
\citep[see e.g.][]{Pitman:2006,Bertoin:2006}.

\subsection{Exchangeable sequences vs exchangeable partitions}

A partition $\part$ subdivides $\bbN$ into subsets, called blocks, such that each element of $\bbN$
is contained in one and only one block. The blocks of $\pi$ can be ordered uniquely by their
smallest elements, and then numbered consecutively; 
denote the set of ordered partitions of $\bbN$ so defined by $\mathfrak{P}$.
A partition can be identified with the sequence of block labels of its elements, \ie the sequence $x(\pi)$ with
\begin{equation*}
  x_n(\pi) = \text{ number of the block containing $n$ in $\pi$ }\;.
\end{equation*}
Any sequence in $\bbN$ thus defines a partition, though not necessarily with ordered blocks---the
blocks are ordered iff 
\begin{equation}
  \label{eq:ordered}
  x_n\quad\leq\quad|\braces{x_1,\ldots,x_n}|
  \qquad\text{ for all }n\;.
\end{equation}
Two partitions $\pi$ and $\pi'$ are
isomorphic, ${\pi\cong\pi'}$, if they differ only in the enumeration of their blocks.
A sequence representing a possibly unordered partition can be turned into an isomorphic, ordered one
by means of the relabeling map
\begin{equation*}
  r(x):=x^{\ast}\qquad\text{ where }\qquad
  x^{\ast}_n:=\begin{cases}
  x^{\ast}_j & \text{ if }x_n=x_j\text{ for some }j<n\\
  |\braces{x_1^{\ast},\ldots,x_{n-1}^{\ast}}|+1
  & \text{ otherwise }
  \end{cases}\;.
\end{equation*}
Then ${\mathfrak{P}=r(\mathbb{N}^{\infty})=\bbN^{\infty}\!/\!\cong}$.
The permutation group $\FSym$ naturally acts on $\bbN^{\infty}$ as
\begin{equation}
  \label{eq:action:sequences}
  T(\phi,x)=(x_{\phi(1)},x_{\phi(2)},\ldots)\;,
\end{equation}
and on $\mathfrak{P}$ by acting on the underlying set $\bbN$, which defines an action $T'$ as
\begin{equation}
  \label{eq:action:partitions}
  i\text{ in block }j\text{ of }T'(\phi,\pi)
  \qquad\Leftrightarrow\qquad
  \phi(i)\text{ in block }j\text{ of }\pi\;.
\end{equation}
For simplicity, we write ${\phi(\pi)=T'(\phi,\pi)}$
and ${\phi(x)=T(\phi,x)}$.
The two actions correspond as
\begin{equation*}
  x(\phi(\pi))=r(\phi(x(\pi)))\;.
\end{equation*}
In other words, if a sequence $x$ is an ordered representation of ${\pi\in\mathfrak{P}}$, then
${r(\phi(x))}$ is an ordered representation of ${\phi(\pi)}$, and the maps ${x\mapsto r(\phi(x))}$
define an action of $\FSym$ on the subset of sequences in $\bbN^{\infty}$ that satisfy \eqref{eq:ordered}.

A random sequence ${X}$ in $\bbN^{\infty}$ and a random partition $\Part$ in $\mathfrak{P}$ are
respectively called \emph{exchangeable} if ${\phi(X)\equdist X}$ and 
${\phi(\Part)\equdist\Part}$, for all ${\phi\in\FSym}$. Kingman's theory of exchangeable partitions
\citep{Kingman:1978:2} exposes a subtle difference between the two cases: Suppose $X$ is a random sequence
such that ${\pi(X)\in\mathfrak{P}}$ almost surely. Then
\begin{equation*}
  X \text{ exchangeable sequence }
  \quad\Rightarrow\quad
  \pi(X) \text{ exchangeable partition, }
\end{equation*}
but the converse is not true: Since $X$ is exchangeable, de Finetti's theorem implies
there is a random probability measure $\mu$ on $\bbN$ such that ${X_1,X_2,\ldots|\mu\simiid\mu}$.
Hence, every value that occurs in $X$ also reoccurs infinitely often with probability 1, 
and every block of the partition $\pi(X)$ is almost surely of infinite size. Kingman's representation
theorem shows that an exchangeable partition can contain two types of blocks, infinite blocks and singletons.
Aldous' proof of Kingman's theorem \citep[e.g.][]{Bertoin:2006} shows that every exchangeable partition $\Pi$ can be
encoded as an exchangeable sequence $X'$ in ${[0,1]}$: The numbers $i$ and $j$ are in the same
block of $\Pi$ iff ${X'_i=X'_j}$. Again by de Finetti's theorem, there is a random probability
measure $\nu$ on $[0,1]$ such that ${X'_1,X'_1,\ldots|\nu\simiid\nu}$. Atoms of $\nu$ account
for infinite blocks of ${\Pi=\pi(X')}$, whereas the continuous component of $\nu$ generates
singleton blocks. Thus, de Finetti's theorem on $[0,1]$ subsumes Kingman's theorem; de
Finetti's theorem on $\bbN$ does not. The latter fails precisely for those partitions
that contain singleton blocks.

\subsection{Subsampling sequences}

Now consider both cases from a subsampling perspective. To generate exchangeable sequences,
we define an algorithm with input ${y\in\bbN^{\infty}}$ as:
\begin{algorithm}
  \label{alg:sequence:sampling}
  \begin{tabular}{rl}
    \text{{\small i.})} & \text{Select $k$ indices ${J_1,\ldots,J_k}$ in $[n]$ uniformly without replacement.}\\
    \text{{\small ii.})} & \text{Extract the subsequence ${\Alg{n}{k}(y)=(y_{J_1},\ldots,y_{J_n})}$.}\\
  \end{tabular}
\end{algorithm}
The algorithm can be represented as a random transformation, using the action \eqref{eq:action:sequences},
\begin{equation*}
  \Alg{n}{k}(y)=\restk{\Phi_n(y)}\qquad\text{ for }\Phi_n\text{ uniform on }\Sym{m}\;.
\end{equation*}
A sequence $y$ has prefix densities under this algorithm if, for each ${k\in\bbN}$ and every finite
sequence $(x_1,\ldots,x_k)$ in $\bbN$, the scalar
\begin{equation*}
  p(x_1,\ldots,x_k):=
  \lim_{n\rightarrow\infty}
  \frac{|\braces{\phi\in\Sym{n}:\textrestk{\phi(y)}=(x_1,\ldots,x_k)}|}{|\Sym{n}|}
\end{equation*}
exists. Let ${\subIN\subset\bbN^{\infty}}$ be the set of all sequences $y$ satisfying this condition.
Then \cref{result:direct:union:sampler} holds on $\subIN$, which shows that $\AlgInf{\infty}(y)$
is an exchangeable random sequence in $\bbN$. Moreover, the algorithm is idempotent, which implies 
\begin{equation}
  \label{eq:factorization:sequence:densities}
  p(x_1,\ldots,x_k)=p(x_1)\cdots p(x_k)\;.
\end{equation}
By Fatou's lemma, 
\begin{equation*}
  \bar{p}:=\sum_{m\in\bbN|p(m)\text{ exists}}p(m)
  \qquad\text{ satisfies }\qquad
  \bar{p}\;\leq\; 1\;.
\end{equation*}
If $X_{n,1}$ denotes the first entry of the sequence output by the sampler on input of length $n$, 
then ${p(m)=\lim_nP\braces{X_{n,1}=m}}$, and as these events are mutually exclusive for different $m$,
\begin{equation}
  \label{eq:sequence:densities:additive}
  y\in\bbN^{\infty}\text{ has prefix densities }
  \quad\Leftrightarrow\quad
  \bar{p}=1
  \quad\Leftrightarrow\quad
  \sum_{m\in\bbN}p(m)=1\;.
\end{equation}
We note every ergodic exchangeable sequence can be obtained in this way: For ${y\in\subIN}$ fixed,
${\AlgInf{\infty}(y)}$ is ergodic, with de Finetti measure ${\sum_{m\in\bbN}p(m)\delta_m}$. It is not
hard to see that for every choice of scalars ${p(m)}$ with ${\sum_mp(m)=1}$, one can construct
a sequence ${y\in\subIN}$ that yields these scalars as prefix limits. By \cref{theorem:ergodic:decomposition},
all exchangeable sequences can be obtained by randomization,
as ${\AlgInf{\infty}(Y)}$, for some (not necessarily 
exchangeable) random sequence $Y$. 
The factorization \eqref{eq:factorization:sequence:densities}
can be read as a combinatorial explanation of de Finetti's theorem on $\bbN$.

\subsection{Subsampling partitions}

For an input partition ${\pi\in\mathfrak{P}}$, define subsampling as:
\begin{algorithm}
  \label{alg:partition:sampling}
  \begin{tabular}{rl}
    \text{{\small i.})} & \text{Select $k$ indices ${J_1,\ldots,J_k}$ in $[n]$ uniformly without replacement.}\\
    \text{{\small ii.})} & \text{Extract the block labels ${(x_{J_1}(\pi),\ldots,x_{J_k}(\pi))}$.}\\
    \text{{\small iii.})} & 
    \text{Report the ordered sequence ${\Alg{n}{k}(\pi)=r(x_{J_1}(\pi),\ldots,x_{J_k}(\pi))}$.}
  \end{tabular}
\end{algorithm}
The algorithm can again be represented as a random transformation,
\begin{equation*}
  \Alg{n}{k}(\pi):=\restk{\Phi_n(\pi)}=\restk{r(\Phi_n(x(\pi)))}\qquad\text{ for }\Phi_n\sim\text{Uniform}(\Sym{n})\;.
\end{equation*}
The second identity shows that it can be reduced to \cref{alg:sequence:sampling}: 
Transform $\pi$ to the sequence $x(\pi)$, apply \cref{alg:sequence:sampling}, and 
reorder the output sequence.

If the input partition $\pi$ is regarded as a sequence ${y=x(\pi)}$, its prefix densities
under \cref{alg:partition:sampling} clearly exist if they do under \cref{alg:sequence:sampling},
and hence if \eqref{eq:sequence:densities:additive} holds. That condition is sufficient, not necessary:
Suppose for every element ${m\in\mathbb{N}}$, the limit $p(m)$ as above exists for $x(\pi)$,
and let ${N_0(\pi)\subset\bbN}$ be the set of all $m$ with ${p(m)=0}$. Require that
\begin{equation*}
  p_0:=\lim_{n\rightarrow\infty}
  \frac{|\braces{j\leq n: x_j(\pi)\in N_0(\pi)}|}{n}\quad\text{ exists, hence }
  p_0+\bar{p}=p_0+\sum_mp(m)=1\;.
\end{equation*}
If ${p_0>0}$, the sequence $x(\pi)$ does not have prefix densities
under \cref{alg:sequence:sampling}, since ${\bar{p}<1}$.
In contrast, \cref{alg:partition:sampling} has prefix densities even for ${p_0>0}$,
since it does not resolve differences between elements of $N_0(\pi)$.

To make this more precise from the algorithmic perspective, assume 
\cref{alg:sequence:sampling} is applied to the sequence representation $x(\pi)$ of $\pi$,
and similarly regard the partition output by \cref{alg:partition:sampling} as a sequence.
Fix an index $i\leq k$, and suppose the $i$th entry of the sequence generated 
in step (ii) of either algorithm takes a value in $N_0(\pi)$. We can distinguish two cases:
A specific value ${m\in N_0(\pi)}$ is reported, or we only report whether or not the value is in $N_0(\pi)$.
These correspond to the events
\begin{equation*}
  A_{im}:=\braces{ x_{J_i}(\pi)=m} \quad\text{ for some }m\in N_0(\pi)
  \qquad\text{ and }\qquad
  A_{i0}:=\braces{ x_{J_i}(\pi)\in N_0(\pi)}\;.
\end{equation*}
Under \cref{alg:sequence:sampling}, the events ${A_{im}}$ are
observable, and hence measurable in $\sigma(\Alg{n}{k}(y))$ for every ${i\leq k\leq n}$. 
Since $p_0$ is the probability of $A_{i0}$ under the limiting output distribution
${P_y}$, where ${y=x(\pi)}$, we have
\begin{equation*}
  p_0=P_{x(\pi)}(A_{i0})=\sum_{m\in N_0(\pi)}P_{x(\pi)}(A_{im})=\sum_{m\in N_0(\pi)}p(m)\;.
\end{equation*}
Since ${p(m)=0}$ by assumption and ${N_0(\pi)}$ is countable, that excludes the case ${p_0>0}$ for 
\cref{alg:sequence:sampling}. Under \cref{alg:partition:sampling}, the
events ${A_{im}}$ are \emph{not} measurable, since they are masked by step (iii), and so
${p_0>0}$ does not lead to contradictions.
In summary, the set of sequences which have prefix densities under \cref{alg:sequence:sampling} is
\begin{equation*}
  \subIN=\bigbraces{y\in\bbN^{\infty}\,\vert\, p(m) \text{ exists for all }m\in\bbN\text{ and }\sum_{m}p(m)=1}\;.
\end{equation*}
\cref{alg:partition:sampling} has prefix densities on the strictly larger set 
\begin{equation*}
  \subIN'=\bigbraces{y\in\bbN^{\infty}\,\vert\, p(m) \text{ exists for all }m\in\bbN\text{ and }\sum_{m}p(m)\leq 1}\;,
\end{equation*}
and the input sequences in ${\subIN'\setminus\subIN}$ are precisely those that generate ergodic exchangeable
partitions with singleton blocks. \cref{result:direct:union:sampler} holds on $\subIN'$ for
\cref{alg:partition:sampling}, which in particular implies idempotence of the algorithm, and this
property is indeed used by Kingman: 
Note identity (1.3) in \citep{Kingman:1978:2} is precisely idempotence, if not by the same name.

\section{Selecting edges independently}
\label{sec:sampling:edges}

We next consider undirected input graphs $y$, with vertex set $\bbN$ and infinitely
many edges, represented as sequences ${y=((i_k,j_k)_{k\in\bbN})}$ of edges as described in 
\cref{sec:spaces}. The sequence representation makes this case similar to
the sequences and partitions in the previous section. In analogy to the set $\mathfrak{P}$
of partitions represented by sequences with ordered labels, we define
\begin{equation*}
  \mathcal{G}:=\braces{y\in(\bbN^2)^{\infty}\,\vert\,
    i_k<j_k \text{ for all }k\in\bbN\text{ and }
    (i_1,j_1,i_2,j_2,\ldots) \text{ satisfies \eqref{eq:ordered}}}\;.
\end{equation*}
These are those undirected graphs on $\bbN$ that have infinitely many edges, and whose
vertices are labeled in order of appearance in the sequence. The set contains both
simple graphs (if each edge occurs only once), and multigraphs.
We select edges uniformly by applying a version of \cref{alg:partition:sampling}
to the edge sequence $y$. That requires a suitable adaptation of the relabeling map $r$:
For any finite sequence ${x=((i_1,j_1),\ldots,(i_n,j_n))}$ of edges, define
\begin{equation*}
  r'(x):=\text{swap}(r(x))
  \quad\text{ where }\quad
  \text{swap}\bigl((i_k,j_k)_{k\leq n}\bigr)
  =(\min\braces{i_k,j_k},\max\braces{i_k,j_k})_{k\leq n}\;.
\end{equation*}
In words, apply the relabeling map $r$ for sequences to the sequence of edges (read as a 
sequence of individual vertices, rather than of pairs), and then swap each pair of vertices
if necessary to ensure ${i_k\leq j_k}$. The sampling algorithm is defined as follows:
\begin{algorithm}
  \label{alg:edge:sampling}
  \begin{tabular}{rl}
    \text{{\small i.})} & \text{Select $k$ indices ${J_1,\ldots,J_k}$ in $[n]$ uniformly without replacement.}\\
    \text{{\small ii.})} & \text{Extract the sequence ${((i_{J_1},j_{J_1}),\ldots,(i_{J_n},j_{J_n}))}$.}\\
    \text{{\small iii.})} & \text{Report the relabeled graph ${\Alg{n}{k}(y):=r'((i_{J_1},j_{J_1}),\ldots,(i_{J_k},j_{J_k}))}$.}\\
  \end{tabular}
\end{algorithm}
As in the partition case, the algorithm can be represented by a permutation followed by an application of $r'$:
Define
\begin{equation}
  \label{eq:action:permute:edges}
  T\bigl(\phi,((i_k,j_k)_k)\bigr):=r'((i_{\phi(k)},j_{\phi(k)})_k)
  \qquad\text{ for }\phi\in\FSym\;.
\end{equation}  
Then $T$ is a prefix action of $\FSym$ on the set $\mathcal{G}$, and leaves the set invariant,
${T(\mathcal{G})=\mathcal{G}}$. \cref{alg:edge:sampling} satisfies 
${\Alg{n}{k}(y)\equdist\textrestk{T(\Phi_n,y)}}$ if ${\Phi_n}$ is uniformly distributed on ${\Sym{n}}$.
For a sequence $x_k$ of $k$ edges, labeled in order, the prefix density is
\begin{equation*}
  \hom_{x_k}(y)
  =
  \lim_{n\rightarrow\infty}
  \frac{
    \text{\small\# of $k$-edge subgraphs of $\textrestn{y}$ isomorphic to }x_k
  }{
    \text{\small\# of $k$-edge subgraphs of $\textrestn{y}$}
  }\;,
\end{equation*}
provided the limit exists.

\subsection{Sampling edges of simple graphs}

Suppose first the input graphs are simple: We choose the input set $\subIN$ as
\begin{equation*}
  \subIN:=\braces{y\in\mathcal{G}\,\vert\, (i_k,j_k)\neq(i_l,j_l)\text{ whenever }k\neq l}\;.
\end{equation*}
Since $T$ does not change the multiplicities of edges, $\subIN$ is a measurable, $T$-invariant
subset of $\mathcal{G}$, and $T$ restricted to $\subIN$ is again a prefix action of $\FSym$.
Define the \kword{limiting relative degree} of vertex $i$ as
\begin{equation*}
  \rdeg(i,y)
  \quad:=\quad
  \lim_{n\rightarrow\infty}\quad
  \frac{1}{2n}\deg(i,\textrestn{y})\;.
\end{equation*}
If ${\rdeg(i,y)=0}$ for all vertices $i$ in $y$---for example, if all degrees are finite---the 
limiting probability of any vertex reoccuring in ${\AlgInf{\infty}(y)}$ is zero. If so,
sequences of isolated edges occur almost surely, and ${t_{x_k}(y)=1}$ if ${x_k=((1,2),(3,4),\ldots,(2k-1,2k))}$,
or $0$ otherwise.

In analogy to \cref{sec:partitions}, define
\begin{equation*}
  \bar{p}(y) = 
  \sum_{i\in\bbN}\rdeg(i,y)\;.
\end{equation*}
If ${\bar{p}(y)>0}$, there is at least one vertex $i$ with ${\rdeg(i,y)>0}$, and as for partitions,
${\bar{p}(y)\leq 1}$. An infinite simple graph $y$ with ${\bar{p}(y)>0}$ satisfies
\begin{equation*}
  \lim_{n\rightarrow\infty}\quad
  \frac{1}{n}
  \sum_{k\leq n}\mathbb{I}\braces{\rdeg(i_k,y)\rdeg(j_k,y)>0}
  \quad=\quad
  0\;,
\end{equation*}
\ie those edges contained in the induced subgraph on all vertices with positive relative
degree make up only an asymptotically vanishing fraction of all edges in the graph. Roughly speaking,
many more vertices than those with positive relative degree are needed to account for such large
degrees on some vertices.

The limiting probability that step (ii) of the algorithm selects an edge connecting two vertices with 
positive relative degree is hence zero. 
Every edge $(i_k,j_k)$ thus has at least one terminal vertex with ${\rdeg=0}$, and by definition of the
relabeling map $r'$, this implies $j_k$ always corresponds to a vertex with ${\rdeg=0}$. 
It then follows by comparison to \cref{alg:partition:sampling} that the sequence ${(i_1,i_2,\ldots)}$
represents an exchangeable random partition, and all arguments in \cref{sec:partitions} apply,
with ${p(m):=\rdeg(m,y)}$ and ${p_0:=1-\bar{p}(y)}$. By \cref{lemma:main},
the normalized number of edges connected to a vertex converges to $\rdeg$. We have shown:
\begin{corollary}
  If the input graph ${y\in\mathcal{G}}$ of \cref{alg:edge:sampling} is simple, each connected component is
  of $\AlgInf{\infty}(y)$ is either a star or an isolated edge. Let $D_k(i)$ be the number
  of edges in the $i$th-largest star $\AlgInf{k}(y)$, and $D_k(0)$ the number of isolated edges. Then
  \begin{equation}
    D_k(0)\rightarrow 1-\bar{p}(y)
    \qquad\text{ and }\qquad
    D_k(i)\xrightarrow{k\rightarrow\infty} \Delta_i
    \qquad\text{ almost surely}\;,
  \end{equation}
  where $\Delta_i$ is $i$th-largest limiting relative degree of $y$.
\end{corollary}
Since the algorithm essentially extracts an exchangeable partition from $y$ whose block sizes
correspond to the relative limiting degrees, but no other information, 
a statistic of $y$ can be estimated using \cref{alg:edge:sampling} if and only if it
is a measurable function of the limiting degree sequence of $y$. In conclusion:
\emph{Uniform sampling of edges from large simple graphs is useful if and only if the objective is
  to estimate a property of the degree sequence.}

\subsection{Sampling edges of multigraphs}
\label{sec:edge:exch:multigraphs}

Now suppose the input $y$ is a multigraph, \ie edges may reoccur in the sequence ${y=((i_k,j_k)_k)}$.
The \kword{limiting relative multiplicity} of a multiedge is 
\begin{equation*}
  \rmult(i,j,y):=\lim_{n\rightarrow\infty}
  \frac{1}{n}\sum_{k\leq m}\mathbb{I}\braces{(i,j)=(i_k,j_k)}\;.
\end{equation*}
We assume there exists at least one edge with ${\rmult(i,j,y)>0}$.
In analogy to $\bar{p}(y)$ above, we define a similar quantity in terms of multiplicities as
\begin{equation*}
\bar{\mu}(y):=\sum_{i<j}\rmult(i,j,y)\;.
\end{equation*}
Since we are effectively sampling from a sequence, the discussion in \cref{sec:partitions}
shows we have to distinguish input graphs with ${\bar{p}(y)=1}$ and ${\bar{p}(y)<1}$.
The next few results apply only to the former, \ie we consider the input set
\begin{equation*}
  \subIN:=\braces{y\in\mathcal{G}\,\vert\,\bar{p}(y)=1}\;.
\end{equation*}

\begin{corollary}
  If \cref{alg:edge:sampling} applied to input graph with ${\bar{\mu}(y)=1}$, 
  then ${\AlgInf{\infty}(y)}$ is ``edge-exchangeable'' in the sense of 
  {\rm \citep{Crane:Dempsey:2016:1}} and {\rm \citep{Cai:Campbell:Broderick:2016:1}}, 
  and all edge-exchangeable graphs can 
  be obtained in this manner, with $y$ randomized if the output graph is non-ergodic.
\end{corollary}
That follows immediately from the definition of edge-exchangeable graphs as exchangeable sequences
of pairs ${i<j}$ in $\bbN$ in \citep{Crane:Dempsey:2016:1}, and our analysis of \cref{alg:sequence:sampling}
above. We also conclude from \cref{sec:partitions} that \cref{result:direct:union:sampler} and
\cref{result:estimate:f:of:y} are applicable, hence:
\begin{corollary}
  \label{corollary:sampling:edges:simple}
  On $\subIN$, the action \eqref{eq:action:permute:edges} is a symmetry of \cref{alg:edge:sampling}, and the algorithm
  is idempotent and resolvent on $\subIN$. Moreover:
  \begin{enumerate}
    \renewcommand{\labelenumi}{(\roman{enumi})}
  \item 
  If two input graphs $y$ and $y'$ have identical prefix densities,
  then ${P_y=P_{y'}}$; otherwise, there exists a $T$-invariant Borel set ${A\subset\OUT}$ such that
  ${P_y(A)=1}$ and ${P_{y'}(A)=0}$.
  \item 
  Let ${f}$ be a function in ${\Lone(P_y)}$, and
  $(f_k)$ a sequence of Borel functions ${f_k:\OUT_k\rightarrow\mathbb{R}}$ with
  ${f_k(\textrestk{x})\rightarrow f(x)}$ as ${k\rightarrow\infty}$, for all $x$ outside
  a $P_y$-null set. Then almost surely
  ${\frac{1}{k!}\sum_{\phi\in\Sym{k}}f(T_{\phi}(S_k(y))\rightarrow f(y)}$
  as ${k\rightarrow\infty}$.
  \end{enumerate}
\end{corollary}
Statement (ii) implies in particular that multiplicities converge: If $\nu_k(y)$ denotes the
$k$th-largest limiting relative multiplicity in a graph $y$, then
\begin{equation*}
  \nu_k(\AlgInf{\infty}(y))=\nu_k(y)\qquad\text{ almost surely.}
\end{equation*}
The behavior of \cref{alg:edge:sampling} changes significantly compared to simple input graphs:
${\AlgInf{\infty}(y)}$ discovers an edge location $(i,j)$ in $y$ if and only 
if ${\rmult(i,j,y)>0}$, and a vertex $i$ in $y$ if and only if it is the terminal vertex of such an edge.
Thus, the output distribution is now governed entirely by edge multiplicities, rather than by vertex degrees
as in the simple case. 

\begin{remark}
If one allows more generally input graphs $y$ with ${\bar{\mu}(y)\leq 1}$, the situation becomes
more complicated: Singleton blocks in partitions now correspond to edges with $\rmult=0$. In partitions,
singleton blocks can exist only because they are not individually identifiable (see \cref{sec:partitions}),
but edges with $\rmult=0$ become partly identifiable if they share a terminal vertex with an edge
with positive relative multiplicity. 
The following example shows, however, that there are input graphs with ${\bar{\mu}(y)<1}$ whose 
prefix densities exist: Suppose the input graph $y$ of \cref{alg:edge:sampling} is chosen such that 
${\textrestn{y}}$, for $n$ even, is
\begin{equation}
  \label{eq:graph:dbiased:resolution}
  \begin{split}
\begin{tikzpicture}[mybraces]
    \begin{scope}[scale=1]
      \node at (-4,.5) {$\textrestn{y}\qquad=$};
    \node at (-.5,0) {};
    \node[circle,scale=.65,thick,draw,fill=gray!30!white] (u1) at (0,1) {}; 
    \node[circle,scale=.65,thick,draw,fill=gray!30!white] (u2) at (-1.5,1) {}; 
    \node[circle,scale=.65,thick,draw,fill=gray!30!white] (v2) at (0,0) {};    
    \node[circle,scale=.65,thick,draw,fill=gray!30!white] (v3) at (1,0) {};    
    \node[circle,scale=.5,thick] (v5) at (3,0) {$\cdots$};      
    \node[circle,scale=.65,thick,draw,fill=gray!30!white] (v6) at (5,0) {};    
    \node at ($(v2)+(0,-.5)$) {};
    \node at ($(v3)+(0,-0.5)$) {};
    \node at ($(v4)+(0,-0.5)$) {};
    \node at ($(v5)+(0,-0.5)$) {};
    \node at ($(v6)+(0,-0.5)$) {};
    \draw (u1)--(u2) node[pos=.5,scale=.1,label=above:{$\frac{n}{2}$}] {};
    \draw (u1)--(v2) node[pos=.5,scale=.1,label=left:{\scriptsize $1$}] {};
    \draw (u1)--(v3) node[pos=.4,scale=.1,label=below:{\scriptsize $1$}] {};
    \draw[dotted] (u1)--(v5);
    \draw (u1)--(v6) node[pos=.5,scale=.1,label=above:{\scriptsize $1$}] {};
    \draw[mirrorbrace] ($(v2.center)+(-.1,-.5)$)--($(v6.center)+(.1,-.5)$) 
    node[pos=.5,label=below:{$\frac{n}{2}$ times}] {};
    \end{scope}
\end{tikzpicture}
  \end{split}
\end{equation}
For finite $n$, each draw in step (ii) of the algorithm selects the single ``large'' multiedge with multiplicity
$n/2$ with probability $1/2$, and each of the other edges with probability $\frac{1}{2n}$. Clearly,
the output graph is of the same form---a single edge whose relative multiplicity converges to $1/2$,
and $\sim n/2$ edges with mulitplicity 1.
\end{remark}

\section{Selecting neighborhoods and random rooting}
\label{sec:neighborhoods}

Instead of extracting individual edges or induced subgraphs, one can
draw neighborhoods around given vertices. One such problem, where one extracts
the immediate neighborhoods of multiple vertices, is of great practical interest,
but we do not know its invariance properties. Another, where one extracts a large
neighborhood of a single random vertex, reduces to well-known results, but
turns out be of limited use for ``statistical'' problems.

\subsection{Multiple neighborhoods}

Let $B_k(v)$ denote $k$-neighborhood of vertex $v$ in the graph, \ie 
the ball of radius $k$ around $v$.
\begin{algorithm}
  \label{alg:ego:networks}
  \begin{tabular}{rl}
    \text{{\small i.})} & \text{Select $k$ vertices ${V_1,\ldots,V_k}$ in $\textrestn{y}$ uniformly at random.}\\
    \text{{\small ii.})} & \text{Report the $1$-neighborhoods of these vertices, ${\Alg{n}{k}(y):=\braces{B_1(V_1),\ldots,B_1(V_n)}}$.}\\
  \end{tabular}
\end{algorithm}
The networks literature often refers to the 1-neighborhood as the \emph{ego network} of $v$.
Data consisting of a collection of such ego networks is of considerable practical relevance, 
and it is hence an interesting question what the invariance properties of \cref{alg:random:rooting} 
are; at present, we have no answer. This question is related to a number of open problems; for example,
how many vertices need to be sampled to obtain a good reconstruction of an input graph,
known as the \emph{shotgun assembly problem} \citep{Mossel:Ross:2015:1}.

\subsection{The Benjamini-Schramm algorithm}
\label{sec:benjamini:schramm}

A different strategy is to report not small neighborhoods of many vertices, but a large neighborhood of a single 
vertex; in this case, a number of facts are known. A \kword{rooted graph} is a pair ${x_{\ast}=(x,v)}$, where $x$ is a graph and $v$ a distinguished vertex in $x$, the \kword{root}.
Let $\OUT$ be the space of undirected, connected, rooted graphs with finite vertex degrees on the vertex set $\bbN$.
For ${x_{\ast}\in\OUT}$, define $\textrestk{x_{\ast}}$ as the induced subgraph on those vertices
with distance ${\leq n}$ to the root of $x$; that is, the ball ${B_n(v,x)}$ of radius $n$ in $x$ centered at the
root $v$. Then ${\OUT_n=\textrestk{\OUT}}$ is the set of finite, rooted, undirected,
connected graphs with diameter ${2n+1}$ (rooted at a vertex ``in the middle'').
As $\OUT_n$ is countably infinite, $\OUT$ is procountable but not compact. 
It is separable, and hence almost discrete, by \cref{lemma:procountable}. 
\begin{algorithm}
  \label{alg:random:rooting}
  \begin{tabular}{rl}
    \text{{\small i.})} & \text{Select a root vertex $V$ in $\textrestn{y}$ uniformly at random.}\\
    \text{{\small ii.})} & \text{Report the ball of radius $k$ centered at the root, ${\Alg{n}{k}(y):=B_k(V)}$}\\
  \end{tabular}
\end{algorithm}
\citet{Benjamini:Schramm:2001:1} introduced this algorithm to define a notion of convergence of graphs:
A graph sequence ${(y^{(k)})}$ in $\subIN$ is \emph{locally weak convergent} to $y$ if 
${\AlgInf{k}(y^{(k)})\rightarrow\AlgInf{k}(y)}$
in distribution as ${k\rightarrow\infty}$, for every ${n\in\bbN}$.

\citet{Aldous:Lyons:2007:1} have studied invariance properties of such limits, and their
work provides---from our perspective---a symmetry analysis of \cref{alg:random:rooting}.
An \kword{automorphism} $\phi$ of a graph $x$, with vertex set $\mathbf{V}$ and edge set $\mathbf{E}$, is a 
pair of bijections ${\phi_{\ind{V}}\!:\mathbf{V}\rightarrow\mathbf{V}}$ and 
${\phi_{\ind{E}}\!:\mathbf{E}\rightarrow\mathbf{E}}$
that cohere in the sense that, for each edge ${e\in\mathbf{E}}$, $\phi_{\ind{V}}$ maps terminal vertices of $e$
to the terminal vertices of ${\phi_{\ind{E}}(e)}$.
Let $\Aut(x)$ denote the set of automorphisms of $x$; clearly, it is a group. 
Let $\fam$ be the family of all maps from ${\mathbf{V}\times\braces{i,j\in\bbN|i<j}}$ to itself.
Then ${\Aut(x)\subset\fam}$. For each ${\phi\in\fam}$, define
\begin{equation*}
  t_{\phi}:\OUT\rightarrow\OUT
  \qquad\text{ as }\qquad 
  t_{\phi}(x):=
  \begin{cases}
    \phi(x) & \text{ if }\phi\in\Aut(x)\\
    x & \text{ otherwise }
  \end{cases}
\end{equation*}
Then the set ${\group:=\braces{t_{\phi}|\phi\in\fam}}$ is a group of measurable bijections of $\OUT$.
\citet{Aldous:Lyons:2007:1} call a probability measure on $\OUT_{\ast}$ \kword{involution invariant}
if it is $\mathbf{G}$-invariant. Involution invariance can be understood as a form of stationarity:
Draw a rooted graph $X$ with root $V$ at random. Perform a single step of simple random walk: 
By definition, the root has at least one and at most
finitely many neighbors. Choose one of these neighbors uniformly at random, and shift the root to that
neighbor, which results in a random rooted graph $X'$. Then $X$ is involution invariant if ${X\equdist X'}$.
In other words, in an involution invariant graph, the distribution of neighborhoods of arbitrary diameter
around the current root remains invariant under simple random walk on the graph.
Involution invariance is closely related to the concept of \emph{unimodularity}, which can be formulated
on unrooted graphs as a ``mass-transport principle''; we omit details and refer to 
\citep{Benjamini:Schramm:2001:1,Aldous:Lyons:2007:1}. 
\begin{fact*}[\citet{Aldous:Lyons:2007:1}]
  Let ${\AlgInf{\infty}}$ be defined by \cref{alg:random:rooting}, and $Y$ a random element of $\subIN$. 
  Then ${\AlgInf{\infty}(Y)}$ is $\group$-invariant if and only if the random (rootless) graph
  $Y$ is unimodular.
  For every ${y\in\subIN}$, the output distribution ${P_y=\Law(\AlgInf{\infty}(y))}$ is $\group$-ergodic.
\end{fact*}
\noindent Whether all $\group$-invariant measures can be obtained as $P_y$ for some fixed ${y\in\subIN}$, or more
generally as weak limits ${P=\lim_k P_{y^{(k)}}}$ for some sequence ${(y^{k})}$ in ${\subIN}$, is an open problem
\citep{Aldous:Lyons:2007:1}.

\section*{Acknowledgments}
This work has greatly benefited from discussions with
Nate Ackerman, Morgane Austern, Benjamin Bloem-Reddy, Christian Borgs, Jennifer Chayes, 
Cameron Freer (who also suggested the term \emph{idempotent} and pointed
me to \citep{Vershik:2004:1}), Svante Janson,
Daniel M.\ Roy, and Victor Veitch. It has been supported by grant FA9550-15-1-0074 of AFOSR.

\newpage
\appendix

\section{Technical Addenda}
\label{app:addenda}

\subsection{Almost discrete spaces (\cref{sec:spaces})}
\label{app:spaces}

To check whether a procountable space is almost discrete, one has to verify separability. The
combinatorial character of these spaces makes it possible to formulate conditions in terms of the
elements of $\OUT_n$: One is by verifying each finite structure ${x_n\in\OUT_n}$ can be extended to an element
of $\OUT$ in some canonical way, which is the case if there is a map
\begin{equation}
  \label{eq:embedding:map}
  \tau:\cup_{n}\OUT_n\rightarrow\OUT 
  \qquad\text{ with }\qquad
  \textrestn{\tau(x_n)}=x_n \text{ for all }
  n\in\mathbb{N}
  \text{ and }
  x_n\in\OUT_n\;.
\end{equation}
For example, $\tau$ may pad a finite graph $x_n$ with isolated vertices (if the graphs in $\OUT$ are
not required to be connected). 
\begin{lemma}
  \label{lemma:procountable}
  (i) A procountable space is compact if and only if it is profinite. If so, it is almost discrete.
  (ii) Any procountable space admitting a map $\tau$ as in \eqref{eq:embedding:map} is almost discrete.
\end{lemma}

\begin{proof}
(i) Inverse limits of topological spaces are compact iff each factor is compact.
Compact metrizable spaces are Polish. (ii) By definition of the metric $d$, the preimage
of $\textrestn{x}$ under the map $\textrestn{\argdot}$ is the $d$-ball $B_n(x)$
of radius $2^{-n}$ at $x$, so ${\tau(\textrestn{x})\in B_n(x)}$.
Since the metric balls form a base of the topology, the countable set ${\tau(\cup_n\OUT_n)}$ is dense in $\OUT$.
\end{proof}

  Discrete spaces and continua can be distinguished by their abundance of sets that are clopen
  (simultanuously closed and open): In a discrete space, each subset is clopen; in Euclidean space,
  no proper subset is clopen. An almost discrete spaces can contain sets that are not clopen, but clopen sets 
  determine all toplogical properties: There is a clopen base of the topology (the metric balls), and the
  space is hence \emph{zero-dimensional} in the parlance of descriptive set theory.
  Informally, an almost discrete space is as close
  to being discrete as an uncountable space can be if it is Polish.
  The following lemma adapts standard results on weak convergence and inverse limits to this setting:
\begin{lemma}
  \label{lemma:procountable}
  Let $\OUT$ be almost discrete. (i) A
  sequence $(P_i)$ of probability measures on $\OUT$ converges weakly to a probability measure $P$ if
  and only if ${P_i(B)\rightarrow P(B)}$ for all metric balls $B$ of finite radius.
  (ii) Let $\OUT'$ be a topological subspace of $\OUT$, 
  and $P_n$ a probability measure on $\OUT_n$ for each ${n\in\mathbb{N}}$.
  Then there exists a probability measure $P$ on $\OUT'$ that satisfies ${\textrestn{P}=P_n}$ for 
  all $n$ if an only if, for all $n$,
  \begin{equation*}
      \textrestn{P_{n+1}}=P_n 
      \qquad\text{ and }\qquad
      \text{if }P_n(x_n)>0\text{ then }\textrestn{x}=x_n\text{ for some }x\in\OUT'\;.
  \end{equation*}
  If $P$ exists, it is tight on $\OUT'$.
\end{lemma}

\subsection{Separability of transformation families}

\label{sec:symmetry:addenda}

Whether a given family $\fam$ 
is separable (cf.\ \cref{sec:symmetry}) can be established by a number of criteria.
A general sufficient condition is:
\emph{If ${\fam_0\subset\fam}$ is dense in the topology of point-wise convergence on $\OUT$,
  it is separating} \citep{Farrell:1962:1}.
If $\fam$ is a group action, a sufficient condition can be formulated in terms of the group:
\emph{If $\fam$ is an action of a group $\group$, and if there exists a second-countable, locally 
  compact topology on $\group$ that makes the action measurable, then
  $\fam$ is separable} \citep{Farrell:1962:1}.
Since a second-countable, locally compact group is Polish, each orbit of the action is measurable.
Countable groups have countable orbits, and are trivially separable, but that does not imply
that a separable group action has countable orbits---the infinite symmetric group
(the set of all bijections of $\bbN$) is a counterexample. There is, however, the following remarkable converse,
the Feldman-Moore theorem \citep{Feldman:Moore:1977:1}:
\emph{If $\Pi$ is a partition of a standard Borel space $\OUT$ into countable sets,
  there is a  measurable action of a countable group on $\OUT$ whose orbits
  are the constituent sets of $\Pi$}.
In particular, if each
orbit of ${\fam}$ is countable, then $\fam$ is separable.

\subsection{Left-inversion of prefix densities}

The proof of \cref{result:estimate:f:of:y} establishes that 
${f(\AlgInf{\infty}(y))=f(y)}$ almost surely, by constructing a map
${f':[0,1]^{\infty}\rightarrow\mathbb{R}}$ that factorizes $f$ as ${f=f'\circ\mathbf{\hom}}$ outside
a null set. This map $f'$ depends on the law $P_y$, or more generally on 
$\Law(Y)$ if $y$ is randomized. 
A measure-dependent result suffices for \cref{result:estimate:f:of:y},
but from an analytic perspective, it is interesting to ask whether the statement can be strengthened
to be measure-free. That is possible if $\mathbf{\hom}$ preserves measurability of open sets:
\begin{proposition}
  \label{result:prefix:action:selector}
  Let $S$ be a sampling algorithm defined as in \eqref{eq:random:transformation:sampling} by a prefix action $T$
  on a almost discrete space, and require all prefix densities exist. If $\mathbf{\hom}(A)$ is Borel whenever
  $A$ is open, there exists a Borel map 
  ${\sigma:[0,1]^{\infty}\rightarrow\OUT}$ such that ${\sigma\circ\mathbf{\hom}(y)\equivIN y}$ 
  for all $y$.
\end{proposition}

\begin{proof}
  Since $\equivIN$ is a partition of a standard Borel space, it follows from the theory of measurable partitions
  that the map $\sigma$ exists if (1) each equivalence class is a closed set and (2) for every open set $A$,
  the $\equivIN$-saturation ${A^{\ast}:=\braces{y\in\OUT|y\equivIN y'\text{ for some }y'\in A}}$
  is Borel \citep[e.g.][Theorem 12.16]{Kechris:1995}.
  By \cref{result:direct:union:sampler}, the equivalence classes of $\equivIN$ are the fibers of $\mathbf{\hom}$.
  Since $\mathbf{\hom}$ is continuous by \cref{result:continuous:prefix:densities}, the equivalence classes are
  closed. Observe that the saturation of any set $A$ is ${A^{\ast}=\mathbf{\hom}^{-1}\mathbf{\hom}(A)}$.
  Thus, if $\mathbf{\hom}(A)$ is measurable, so is $A^{\ast}$.
\end{proof}

\section{Proofs}

\subsection{Existence of a limiting sampler}

We denote be $\pMeas(\IN)$ the space of probability measures on $\IN$ (topologized by weak convergence).
The proof uses a continuity result of \citet{Blackwell:Dubins:1983:1}: If $\IN$ is Polish,
$\pMeas(\IN)$ the space of probability measures on $\OUT$, 
and ${(\U,\lambda)}$ a standard probability space, there exists a 
mapping ${\rho:\pMeas(\IN)\times\U\rightarrow\IN}$ and
a $\lambda$-null set ${N_{\rho}\subset\U}$ such that 
\begin{equation}
  \label{eq:blackwell:dubins}
  \rho(\argdot,u) \text{ is continuous for all }u\not\in N_{\rho}
  \qquad\text{ and }\qquad
  \Law(\rho(P,U))=P\;,
\end{equation}
where the latter holds for any ${P\in\pMeas(\IN)}$ and any random variable $U$
with law $\lambda$.

\begin{proof}[Proof of \cref{result:alg:limit}]
  Abbreviate ${U:=(U_1,U_2,\ldots)}$ and 
  ${P(n,k,y):=\Law(\Alg{n}{k}(y,U))}$. Since $\OUT$ is almost discrete,
  it is Polish, and so is each space $\OUT_n$. The spaces of $\pMeas(\OUT)$ and $\pMeas(\OUT_n)$
  are hence again Polish in their weak topologies.
  \begin{itemize}
  \item First fix $k$. Since the prefix densities exist, the sequence ${\lim_nP(n,k,y)}$ converges
    weakly to a limit $P(k,y)$. Since $\IN$ is measurable and ${\OUT_k}$ Polish, 
    measurability of each ${y\mapsto P(n,k,y)}$ implies measurability of the limit as a function
    ${y\mapsto P(k,y)}$. 
  \end{itemize}
  Now write the restriction explicitly as a map ${\pr_k(x):=\textrestk{x}}$, and recall it is continuous.
  It induces a map ${\pMeas(\OUT_{k+1})\rightarrow\pMeas(\OUT_k)}$
  on probability measures, as ${P_{k+1}\mapsto P_{k+1}\circ\pr_k^{-1}}$, which is again continuous.
  \begin{itemize}
  \item The distributional limits above preserve projectivity of laws: 
    By \eqref{eq:def:sampler:2}, we have 
    \begin{equation*}
      \pr_kP(n,k+1,y)=P(n,k,y) 
      \qquad\text{ and hence }\qquad
      \pr_kP(k+1,y)=P(k,y)
    \end{equation*}
    by continuity of ${\pr_k}$.
  \item For ${y}$ fixed, the measures ${(P(k,y))_k}$ hence form a projective family, and
    by \cref{lemma:procountable}, there is a probability measure
    $P(y)$ on $\OUT$ satisfying ${\pr_kP(y)=P(k,y)}$ for each $k$. 
  \end{itemize}
  This already suffices to guarantee that any random variable $\AlgInf{\infty}$ with law $P(y)$
  satisfies \eqref{eq:alg:limit}, pointwise in $y$. What remains to be shown is only that this variable can be chosen
  as a measurable function of $(y,u)$. To this end, let ${\rho:\pMeas(\OUT)\times[0,1]\rightarrow\OUT}$ be the
  representation map guaranteed by \eqref{eq:blackwell:dubins}, and $N$ the associated null set. Additionally,
  define ${f:\IN\times[0,1]\rightarrow\pMeas(\OUT)\times[0,1]}$ as ${f(y,u):=(P(y),u)}$. Then set
  \begin{equation*}
    \AlgInf{\infty}(y,u):=\rho(f(y,u))
    \qquad\text{ which implies }
    \AlgInf{\infty}(y,U)\sim P(y)
    \text{ for each }y\;.
  \end{equation*}
  Suppose a function of two arguments, say ${g:\mathbf{Z}_1\times\mathbf{Z}_2\rightarrow\mathbf{Z}_3}$, is 
  continuous in its first argument
  and measurable in the second. If $\mathbf{Z}_1$ is separable metrizable, $\mathbf{Z}_2$ measurable, and $\mathbf{Z}_3$ metrizable, that
  suffices to make $g$ jointly measurable \citep[][Lemma 4.51]{Aliprantis:Border:2006}.
  \begin{itemize}
  \item The restriction of $\rho$ to ${\pMeas(\OUT)\times([0,1]\setminus N)}$ is continuous in its first argument
    and measurable in the second, hence jointly measurable. Since its range $\OUT$ is Polish, it can be extended
    to a measurable function on all of ${\pMeas(\OUT)\times[0,1]}$ \citep[][Corollary 4.27]{Dudley:2002:1}.
  \item Since $[0,1]$ is
    Polish, ${[0,1]\setminus N}$ is separable metrizable. By the same device as above, $f$ is jointly measurable, which
    makes $\rho\circ f$ is measurable as a function ${\IN\times[0,1]\rightarrow\OUT}$.
  \end{itemize}
  Thus, $\AlgInf{\infty}$ is indeed jointly measurable.
\end{proof}

\subsection{Law of large numbers}
\label{sec:proof:lln}

The main ingredient of the proof of \cref{lemma:main}
is the pointwise ergodic theorem of \citet{Lindenstrauss:2001:1}. For 
countable discrete groups, it states:
\begin{theorem}[E. Lindenstrauss]
  Let $\group$ be a countable group, ${(\A_n)}$ a sequence of finite subsets
  of $\group$ satisfying \eqref{eq:Folner}, and $T$ a measurable action of $\group$ on a standard
  Borel space $\OUT$. If $X$ is a random element invariant under $T$, and ${f\in\Lone(X)}$,
  there is $T$-invariant function ${\bar{f}\in\Lone(X)}$ such that 
  ${|\A_k|^{-1}\sum_{\phi\in\A_k}f(T_{\phi}(X))\longrightarrow\bar{f}}$
  almost surely as ${k\rightarrow\infty}$.
\end{theorem}
We need a lemma relating convergence of a sequence of functions ${f_i}$ and of averages $\mu_k(f)$ to
convergence of the diagonal sequence $\mu_k(f_k)$. It involves random measures 
defined on a probability
space $\OUT$ that take values in the set of measures on the same space---that is, measurable
mappings ${\mu:\OUT\rightarrow\pMeas(\OUT)}$. We denote these ${x\mapsto\mu^x}$.
\begin{lemma}
  \label{lemma:as:convergent:net}
  Let ${(\OUT,\borel(\OUT),P)}$ be a standard Borel probability space, and ${f,f_1,f_2,\ldots}$ 
  measurable functions ${\OUT\rightarrow\mathbb{R}}$ such that ${f_i\rightarrow f}$ uniformly
  on some set ${A\in\borel(\OUT)}$. Let ${\mu_1,\mu_2,\ldots}$ be measurable mappings 
  ${\OUT\rightarrow\pMeas(\OUT)}$, such that the limits
  \begin{equation}
    \label{eq:lemma:as:convergent:net}
  \tau_i(x):=\lim_k\mu_k^x(f_i)
  \quad\text{ and }\quad
  \tau(x):=\lim_k\mu_k^x(f)
  \quad\text{ exist for }P\text{-almost all }x\in\OUT\;.
  \end{equation}
  Then there is a conull set ${A'}$ such that 
  ${\mu_k^x(f_k)\rightarrow\tau(x)}$ for all ${x\in A\cap A'}$, and $\tau$ is measurable
  as a function on $\OUT$.
\end{lemma}
For the proof, consider a real-valued net ${(\tau_{ik})_{i,k\in\mathbb{N}}}$, with the product order on
${\mathbb{N}\times\mathbb{N}}$. Recall that such a net is said to converge to a limit $\tau$ if, 
for every ${\varepsilon>0}$, ${|\tau_{ik}-\tau|<\varepsilon}$ for all $i$ exceeding some $i_{\varepsilon}$ and
$k$ exceeding some ${k_{\varepsilon}}$. Recall further that a sufficient condition for convergence is that
(i) the limits ${\lim_k\tau_{ik}}$ exist for all ${i\in\mathbb{N}}$, and (ii) convergence
to the limits ${\lim_i\tau_{ik}}$ holds and is uniform over $k$. If so, the limits commute, and
${\lim_k(\lim_i \tau_{ik})=\lim_i(\lim_k \tau_{ik})=\tau}$. 

\begin{proof}
  Let $N$ be the union of all null sets of exceptions in \eqref{eq:lemma:as:convergent:net}, 
  and ${A':=\OUT\setminus N}$.
  On ${A\cap A'}$, define ${\tau_{ik}(x):=\empavg_k^x(f_i)}$.
  Fix some ${x\in A\cap A'}$. The net $(\tau_{ik}(x))$ converges uniformly over $n$ as ${n\rightarrow\infty}$:
  Let ${\varepsilon>0}$. Uniform convergence of $(f_i)$ implies ${|f_i(z)-f(z)|<\varepsilon}$ for all ${z\in A\cup A'}$ 
  and all $i$ exceeding some $i_0$, and therefore
  \begin{equation}
    |\mu^x_k(f_i)-\mu^x_k(f)|
    \;\leq\;
    \mu^x_k(|f_i-f|)
    \;\leq\; 
    \varepsilon \qquad\text{ for all }k\in\mathbb{N}, i\geq i_0\;.
  \end{equation}
  If $i$ is fixed instead, ${\lim_k \tau_{ik}(x)}$ exists, by \eqref{eq:lemma:as:convergent:net}. 
  The entire net hence converges for ${x\in A\cap A'}$, to some limit $\tau(x)$,
  and extracting the diagonal sequence yields ${\mu_k^x(f_k)\rightarrow \tau(x)}$.
  Since $\tau$ is
  a limit of measurable functions into a metrizable space, it is measurable.
\end{proof}

Recall that ${(f_i)}$ converges \emph{almost uniformly} if, for every
$\delta>0$, convergence is uniform on a set of probability at least ${(1-\delta)}$.
To move between almost sure and almost uniform convergence, we use Egorov's theorem
\citep{Hewitt:Stromberg:1975,vanDerVaart:Wellner:1996:1}:
For a sequence of $\Lone$ functions, almost sure implies almost uniform convergence
\citep[][Theorem 11.32]{Hewitt:Stromberg:1975}, and vice versa
\citep[][Lemma 1.9.2(iii) and Theorem 1.9.6]{vanDerVaart:Wellner:1996:1}.

\begin{proof}[Proof of \cref{lemma:main}]
  By Lindenstrauss' pointwise ergodic theorem above, there exists a $T$-invariant function
  ${\bar{f}\in\Lone(P)}$ such that ${\empavg^X_k(f)\xrightarrow{k\rightarrow\infty}\bar{f}}$
  almost surely. For any invariant set ${B\in\sigmainv}$, we hence have
  \begin{equation}
    \int_B\bar{f}dP=\lim_k\frac{1}{|\A_k|}\sum_{\phi\in\A_k}\int_B f(\phi(x))P(dx)=\int_B fdP\;,
  \end{equation}
  since $P$ is $T$-invariant. Since $\group$ is countable, \cref{theorem:ergodic:decomposition} implies there
  is a random ergodic measure $\xi$ such that 
  ${P[\argdot|\sigmainv]=P[\argdot|\xi]=\xi(\argdot)}$ almost surely.
  Thus,
  \begin{equation}
    \label{eq:proof:main:aux:1}
    \bar{f}\;=\;\mean[f|\xi]\;=\;\xi(f)
    \qquad\text{ and hence }\qquad
    \empavg_k^X(f)\;\longrightarrow\;\xi(f)\quad\text{a.s.}
  \end{equation}

  \noindent\emph{Almost sure convergence.} 
  Consider the sequence $(f_i)$. Since ${f_i\rightarrow f}$ almost everywhere, Egorov's theorem
  implies that for every ${\delta>0}$, there is a set $B_{\delta/2}$ of measure at least ${1-\delta/2}$,
  such that ${f_i\rightarrow f}$ uniformly on $B_{\delta/2}$. 
  Applying \cref{lemma:as:convergent:net} to the functions $f_i$ and random measures ${\mu_k=\empavg_k}$
  shows there is a conull set $A$ such that 
  \begin{equation}
    \label{eq:egorov:1}
    \tau(x):=\lim_k\empavg^x_k(f_k)
    \qquad\text{ exists for all }
    x\in B_{\delta/2}\cap A\;.
  \end{equation}
  Again by Egorov's theorem, one can find a further set $B'_{\delta/2}$ with 
  ${P(B'_{\delta/2})\geq 1-\delta/2}$, such that convergence is even uniform
  on ${B_{\delta/2}\cap B_{\delta/2}'}$. Thus, for every ${\delta>0}$, there is a set $B$
  of probability ${P(B)\geq 1-\delta}$ such that 
  ${\empavg^x_k(f_k)\rightarrow \tau(x)}$ holds uniformly on $B$,
  and hence almost uniformly on $\OUT$.
  As noted above, that implies 
  ${\empavg^X_k(f_k)\rightarrow \tau(X)}$ almost surely, and by \eqref{eq:proof:main:aux:1},
  ${\tau(X)=\xi(f)}$ a.s.

  \noindent\emph{Convergence in the mean.}
  Define ${h_k(x):=\empavg_k^x(g)}$. 
  By \eqref{eq:proof:main:aux:1}, ${h_k(X)}$
  converges in $\Lone$, which makes ${(h_k)}$
  uniformly integrable. There is hence, for any
  ${\varepsilon>0}$, an integrable function ${h_{\varepsilon}\geq 0}$ such that 
  ${\int_{\lbrace|h_k|>h_{\varepsilon}\rbrace}h_k(x)P(dx)<\varepsilon}$ holds for all $k$.
  The same $h_{\varepsilon}$ then also satisfies
  \begin{equation*}
    \int_{\lbrace|\empavg_k(f_k)|>h_{\varepsilon}\rbrace}\empavg^x_k(f_k)P(dx)<\varepsilon
    \qquad\text{ since }\qquad
    |\empavg^x_k(f_k)|
    \leq 
    \empavg^x_k(|f_k|)
    \leq
    \empavg^x_k(g)
    \leq
    |\empavg^x_k(g)|
  \end{equation*}
  for all $k$. That likewise makes
  $({\empavg_k(f_k))_{k\in\mathbb{N}}}$ uniformly integrable. 
\end{proof}

\subsection{Sampling by random transformation}
Prefix actions on almost discrete spaces are continuous:
\begin{proof}[Proof of \cref{lemma:direct:union:continuous:action}]
  Fix ${\phi\in\group_k}$. Then
  ${T_k(\phi,\textrestk{x})=\restk{T_n(\phi,\textrestn{x})}}$ for ${n\geq k}$,
  by hypothesis \eqref{eq:affects:only:finite:prefix}.
  That makes $T(\phi,\argdot)$ the inverse limit of the mappings ${T_n(\phi,\argdot)}$, for ${n\geq k}$.
  Since each space $\OUT_n$ is discrete, each such map is continuous; as an inverse limit of
  continuous maps, $T(\phi,\argdot)$ is again continuous. As this is true for each ${\phi\in\group}$
  by \eqref{direct:union}, and $\group$ is discrete, the assembled map ${T(\argdot,\argdot)}$ is continuous.
\end{proof}

For the proof of \cref{result:direct:union:sampler}, we note an immediate consequence of the 
definition \eqref{eq:affects:only:finite:prefix} of prefix actions:
If $\Phi_n$ is a uniform random element of $\group_n$, for each ${n\in\mathbb{N}}$, then
  \begin{equation}
    \label{dist:eq:prefix:action}
    T(\Phi_n,T(\Phi_m,y))\equdist T(\Phi_{\max\braces{m,n}},y)
    \qquad\text{and}\qquad
    T(\Phi_k,T(\phi,y))\equdist T(\Phi_{k},y)
  \end{equation}
  for every ${\phi\in\group}$ and all sufficiently large $k$.

\begin{proof}[Proof of \cref{result:direct:union:sampler}]
  Since $T$ is (jointly) measurable,
  ${\Alg{n}{k}}$ is jointly measurable, and satisfies \eqref{eq:def:sampler:2} by construction,
  so $\AlgInf{\infty}$
  exists by \cref{result:alg:limit}.
  
  {\noindent\emph{Invariance of $P_y$.}}
  Let ${\phi\in\group}$; we have to show ${\phi(\AlgInf{\infty}(y))\equdist\AlgInf{\infty}(y)}$.
  Choose $k(\phi)$ such that ${\phi\in\group_{k(\phi)}}$. Then for any ${k\geq k(\phi)}$, 
  \begin{equation*}
    \begin{split}
      \restk{T(\phi,\AlgInf{\infty}(y))}
      \stackrel{\tiny \eqref{eq:affects:only:finite:prefix}}{=}
      T_k(\phi,\lim_{n\geq k}\textrestk{T(\Phi_n,y)})
      \stackrel{\tiny\text{continuity}}{=}
      \lim_n T_k(\phi,\textrestk{T_n(\Phi_n,y)})\quad\text{in distribution,}
    \end{split}
  \end{equation*}
  where the second identity holds by continuity of $T_n$ and ${\textrestk{\argdot}}$.
  The limit on the right satisfies
  \begin{equation*}
    \begin{split}
      \lim_n T_k(\phi,\textrestk{T_n(\Phi_n,y)})
      \stackrel{\tiny \eqref{dist:eq:prefix:action}}{=}
      \lim_n \textrestk{T_n(\Phi_n,y)}
      =
      \lim_n \textrestk{T(\Phi_n,y)}
      =
      \restk{\AlgInf{\infty}(y)}\qquad\text{ in distribution.}
    \end{split}
  \end{equation*}
  Whenever ${\phi\in\group_{k(\phi)}}$, we hence have 
  ${\textrestk{T_{\phi}P_y}=\textrestk{P_y}}$ for all ${k\geq k(\phi)}$, and
  since the finite-dimensional distributions ${\textrestk{P_y}}$ completely determine $P_y$, that implies
  ${T_{\phi}P_y=P_y}$.

  {\noindent\emph{Idempotence.}}
  Fix ${k\leq m\leq n}$. By definition of the sampler in \eqref{eq:random:transformation:sampling}, 
  \begin{equation*}
    \Alg{m}{k}(\Alg{n}{m}(y))
    \equdist
    \rest[k]{T_m(\Phi_m,T_n(\Phi_n,y))}
    \equdist
    \rest[k]{T(\Phi_n,y))}
    \equdist
    \Alg{n}{k}(y)\;,
  \end{equation*}
  where the second identity holds by \eqref{dist:eq:prefix:action}. That implies, 
  using \cref{result:alg:limit}, that also
  \begin{equation*}
    \Alg{n}{k}(\AlgInf{k}(y))
    \equdist
    \AlgInf{k}(y)\;,
  \end{equation*}
  so idempotence holds.

  {\noindent\emph{Prefix density vector.}}
  Since the prefix densities exist, the vector $\mathbf{t}(y)$ exists for all ${y\in\subIN}$.
  Fix ${\phi\in\group}$. Since ${\subIN}$ is invariant, ${\mathbf{t}(T_{\phi}(y))}$ is well-defined.
  The vector has one entry for each ${k\in\mathbb{N}}$ and ${x_k\in\OUT_k}$, and by definition of
  prefix densities,
  \begin{equation*}
    t_{x_k}(T_{\phi}(y))
    =
    \lim_{n}\PAlg\braces{\restk{T(\Phi_n,T(\phi,y))}=x_k}
    \stackrel{\tiny \eqref{dist:eq:prefix:action}}{=}
    t_{x_k}(y)\;,
  \end{equation*}
  so indeed ${\mathbf{t}=\mathbf{t}\circ T_{\phi}}$. For a random draw, idempotence implies
  \begin{equation*}
    t_{x_k}(\AlgInf{\infty}(y))
    =
    \lim_{n}\PAlg\braces{\Alg{n}{k}(\AlgInf{n}(y))=x_k}
    =
    \lim_{n}\PAlg\braces{\Alg{n}{k}(y)=x_k}
    =
    t_{x_k}(y)\qquad\text{a.s.,}
  \end{equation*}
  and hence ${\mathbf{t}(\AlgInf{\infty}(y))=\mathbf{t}(y)}$ almost surely.

  {\noindent\emph{Relation between the measures $P_y$.}}
  The above implies ${P_y(\mathbf{t}^{-1}\mathbf{t}(y))=1}$. Thus, if $y$ and $y'$ are such that
  ${\mathbf{t}(y)=\mathbf{t}(y')}$, then $P_y$ and $P_{y'}$ concentrate on the same set.
  Otherwise,
  \begin{equation*}
    \mathbf{t}(y)\neq\mathbf{t}(y')
    \quad\Rightarrow\quad
    \mathbf{t}^{-1}\mathbf{t}(y))\cap \mathbf{t}^{-1}\mathbf{t}(y'))=\emptyset
    \quad\Rightarrow\quad
    P_y(\mathbf{t}^{-1}\mathbf{t}(y')))=0\;.
  \end{equation*}
  Hence, ${\mathbf{t}(y)=\mathbf{t}(y')}$ if and only if ${y\equivIN y'}$. Moreover,
  since ${\textrestk{P_y}(x_k)=\hom_{x_k}(y)}$, that implies ${P_y=P_{y'}}$ if and only if ${y\equivIN y'}$,
  so the algorithm is resolvent. Since ${\mathbf{t}=\mathbf{t}\circ T_{\phi}}$, the set
  ${\mathbf{t}^{-1}\mathbf{t}(y)}$ is invariant, so $P_y$ and $P_{y'}$ are mutually
  singular on $\sigma(\group)$ unless ${y\equivIN y'}$.
\end{proof}

\begin{proof}[Proof of \cref{result:continuous:prefix:densities}]
  Recall that ${\mathbf{\hom}}$ is a map ${\OUT\rightarrow[0,1]^{\cup_k\OUT_k}}$. Denote the
  subvector of densities of prefixes in $\OUT_k$ by
  ${\mathbf{\hom}^{(k)}:y\mapsto (t_{x_k}(y))_{x_k\in\OUT_k}}$. The latter is precisely the law
  ${\Law(\AlgInf{k}(y))}$, represented as a vector of probabilities on the countable set $\OUT_k$,
  \begin{equation*}
    \mathbf{\hom}^{(k)}(y)
    \quad=\quad
    \Law(\AlgInf{k}(y))
    \quad=\quad
    \lim_n\restk{T(\Law(\Phi_n),y)}\;.
  \end{equation*}
  Clearly, ${\mathbf{\hom}}$ is continuous if and only if $\mathbf{\hom}^{(k)}$ is continuous for
  each $k$. Fix $k$, and any sequence ${(y_i)}$ in $\OUT$ with limit $y$, and define the net
  \begin{equation*}
    \alpha_{in}:=\restk{T(\Law(\Phi_n),y_i)}\;.
  \end{equation*}
  Consider row- and column-wise convergence of the net:\\[.3em]
  \noindent\emph{(i)}
  Hold $i$ fixed: Since the prefix densities exist, ${\alpha_{in}}$ converges
  as ${n\rightarrow\infty}$.\\[.3em]
  \noindent\emph{(ii)}
  Hold $n$ fixed:  Since $\OUT$ is almost discrete, ${y_i\rightarrow y}$
  implies, for every ${n\in\mathbb{N}}$, that ${\textrestn{y_i}=\textrestn{y}}$ for all sufficiently
  large $i$. Since $T$ is a prefix action, that in turn means 
  \begin{equation*}
    \textrestk{T(\Law(\Phi_n),y_i)}=\textrestk{T(\Law(\Phi_n),y)}
    \quad\text{ hence }\quad
    \alpha_{in}\;\xrightarrow{i\rightarrow\infty}\; \restk{T(\Law(\Phi_n),y)}\quad\text{ uniformly.}
  \end{equation*}
  Since $(\alpha_{in})$ converges separately in $i$ and $n$, and convergence in $i$ is even uniform, 
  $(\alpha_{in})$ converges as a net to a limit $\alpha$, and 
  ${\lim_{n}\lim_i\alpha_{in} = \lim_i\lim_n\alpha_{in}=\alpha}$. Thus,
  \begin{equation*}
    \lim_i\mathbf{\hom}^{(k)}(y_i)=\lim_i\lim_n\restk{T(\Law(\Phi_n),y_i)}
    =\lim_n\restk{T(\Law(\Phi_n),y)}=\mathbf{\hom}^{(k)}(y)
    \qquad\text{whenever }y_k\rightarrow y\;,
  \end{equation*}
  and ${\mathbf{t}^{(k)}}$ is indeed continuous for every ${k\in\mathbb{N}}$.
\end{proof}

The next lemma is adapted from a standard result on Borel sections
\citep[][A1.3]{Kallenberg:2001}, using the fact that measurable functions between suitable
spaces have a measurable graph \citep[][4.45]{Aliprantis:Border:2006}:
\begin{lemma}
  \label{lemma:measured:section}
  Let ${f:\mathbf{X}\rightarrow\mathbf{Y}}$ be a Borel map from a standard Borel 
  into a second-countable Hausdorff space,
  the latter equipped with its Borel $\sigma$-algebra. Then for every probability measure $P$ on $\mathbf{Y}$, 
  there exists a Borel map ${\sigma_{P}:\mathbf{Y}\rightarrow\mathbf{X}}$ such that
  \begin{equation*}
    f(\sigma_{P}(y))=y\qquad\text{ for }P\text{-almost all }y\in\mathbf{Y}\;.
  \end{equation*}
  The image ${f(\mathbf{X})}$ is measurable in the joint completion
  of the Borel sets on $\mathbf{Y}$ under all probability measures on $\mathbf{Y}$
\end{lemma}

\begin{proof}[Proof of \cref{result:estimate:f:of:y}]
  It suffices to show the randomized case. Abbreviate ${Q:=\Law(Y)}$. The vector $\mathbf{\hom}$ is a 
  map ${\subIN\subset\OUT\rightarrow[0,1]^{\infty}}$. The law $Q$ defines an image measure
  ${Q':=\mathbf{\hom}(P_Y)}$ on $[0,1]^{\infty}$.
  Since ${\subIN\subset\OUT}$ is invariant, it is measurable, so its relative topology makes it 
  a standard Borel space \citep[][13.4]{Kechris:1995}.
  By \cref{lemma:measured:section},
  there exists a map ${\sigma_{Q}:[0,1]^{\infty}\rightarrow\OUT}$ satisfying 
  \begin{equation*}
    \mathbf{\hom}(\sigma_Q(s))=s \;\text{ for }Q'\text{-a.a. } s\in[0,1]^{\infty}
    \quad\text{ and hence }\quad
    \mathbf{\hom}(\sigma_Q(\mathbf{\hom}(y)))=\mathbf{\hom}(y) \;\text{ for }Q\text{-a.a. } y\in\subIN\;.
  \end{equation*}
  By (iii) in \cref{result:direct:union:sampler},
  the equivalence classes of $\equivIN$ are the fibers of $\mathbf{\hom}$. Since $f$ is hence constant on 
  each fiber, the map ${f':=f\circ\sigma_Q}$ satisfies
  ${f=f'\circ\mathbf{\hom}}$ almost surely under $Q$. Hence, by (ii), ${f(\AlgInf{\infty}(y))=f(y)}$ almost surely.
\end{proof}

\bibliography{references}
\bibliographystyle{abbrvnat}

\end{document}